\documentclass[11pt]{amsart}

\usepackage{amsmath}
\usepackage{amsthm} \usepackage{amscd} \usepackage{amssymb}
\usepackage{a4wide} \usepackage{graphicx} \usepackage{multirow}

\usepackage{todonotes} \usepackage{tikz-cd}

\newtheorem{theorem}{Theorem}[section]
\newtheorem{lemma}[theorem]{Lemma}
\newtheorem{corollary}[theorem]{Corollary}
\newtheorem{proposition}[theorem]{Proposition}

\theoremstyle{definition} \newtheorem{definition}[theorem]{Definition}
\newtheorem{example}[theorem]{Example}

\theoremstyle{remark} \newtheorem*{remark}{Remark}
 
\newtheorem*{ack}{Acknowledgments}

\numberwithin{equation} {section} \newcounter{temp}

\makeatletter \def\square{\RIfM@\bgroup\else$\bgroup\aftergroup$\fi
\vcenter{\hrule\hbox{\vrule\@height.6em\kern.6em\vrule}\hrule}\egroup}
\makeatother
 
\DeclareMathOperator{\Ad}{Ad}

\DeclareMathOperator{\Cl}{Cl} \DeclareMathOperator{\codim}{codim}
\DeclareMathOperator{\coker}{coker} 
\DeclareMathOperator{\Def}{Def} \DeclareMathOperator{\depth}{depth}
\DeclareMathOperator{\Der}{Der} \DeclareMathOperator{\diag}{diag}
\DeclareMathOperator{\Ext}{Ext} 
\DeclareMathOperator{\GL}{GL} 
\DeclareMathOperator{\Hom}{Hom} \DeclareMathOperator{\Hilb}{Hilb}
 
\DeclareMathOperator{\inter}{int}
 \DeclareMathOperator{\lcm}{lcm}
 \DeclareMathOperator{\negative}{neg}
 \DeclareMathOperator{\Proj}{Proj}
\DeclareMathOperator{\rank}{rk} \DeclareMathOperator{\Sing}{Sing}
\DeclareMathOperator{\SL}{SL} 
 \DeclareMathOperator{\Spec}{Spec}
 
\DeclareMathOperator{\Sym}{Sym} 

\DeclareMathOperator{\vertices}{vert}

\newcommand{\baseRing}[1]{\ensuremath{\mathbb{#1}}}
\newcommand{\N}{\baseRing{N}} \newcommand{\ZZ}{\baseRing{Z}}
\newcommand{\RR}{\baseRing{R}} \newcommand{\QQ}{\baseRing{Q}}
 
\newcommand{\CC}{\baseRing{C}} \newcommand{\CO}{\mathcal{O}}
\newcommand{\PP}{\baseRing{P}} 
 
\newcommand{\CS}{\mathcal{S}} \newcommand{\CQ}{\mathcal{Q}}
\newcommand{\CF}{\mathcal{F}}
 \newcommand{\CN}{\mathcal{N}}  \newcommand{\CI}{\mathcal{I}}
\newcommand{\PPq}{\baseRing{P}(\mathbf{q})}

\begin{document}

\title{Comparison theorems for deformation functors  via invariant
  theory}

\author{Jan Arthur Christophersen}
\address{Matematisk institutt,  Postboks 1053 Blindern, University of
  Oslo, N-0316 Oslo, Norway} 
\email{christop@math.uio.no}

\author{ Jan O. Kleppe}
\address{Oslo Metropolitan University,
  Faculty of Technology, Art and Design, PB 4 St. Olavs plass, N-0130
  Oslo, Norway} 
\email{JanOddvar.Kleppe@oslomet.no}
\urladdr{https://www.cs.hioa.no/\textasciitilde jank/}

\begin{abstract} We  compare  deformations of algebras to deformations
  of schemes in the setting of invariant theory. Our results
generalize comparison theorems of Schlessinger and the second author
for projective 
schemes. We consider deformations (abstract and embedded)
of  a scheme $X$ which is a good quotient of a quasi-affine
scheme $X^\prime$ by a linearly reductive group $G$ and compare them to
invariant deformations of an affine $G$-scheme containing $X^\prime$ as an open
invariant subset. The main theorems give conditions for when the
comparison morphisms are smooth or isomorphisms. 
\end{abstract}

\maketitle

\section{Introduction} 
Given a projective
scheme $X$ defined by equations $f_1, \dots , f_m \in 
k[x_0, \dots , x_n]$, perturbing the equations in a flat manner so
that they remain homogeneous induce
deformations of $X$. In practice this is often the only way to
construct examples of deformations. In more stringent terms we have a
map between the degree $0$ embedded deformations of the affine cone
$C(X)$ and deformations of $X$ in $\PP^n$. If we take into account trivial
deformations we get a map to the deformations of $X$ as scheme. The
question is, when do we get all deformations this way?

To see precisely what is going on we should compare deformation
functors on Artin rings. If $R = k[x_0, \dots , x_n]$ and $S = R/(f_1,
\dots , f_m)$ then the above describes maps $\Def_{S/R}^0 \to \Hilb_{X/\PP^n}$
where $\Def_{S/R}^0$ is the functor of degree $0$ deformations of $S$
as $R$-algebra and $\Def_{S}^0 \to \Def_{X}$
 where $\Def_{S}^0$ is the functor of degree $0$ deformations of $S$
as $k$-algebra. Generalizing  comparisons theorems of Schlessinger
(\cite{sc:rig}, \cite{sc:onr}) the second author gave in
\cite{kle:def}  exact
conditions for when these maps are isomorphisms. The object of this
paper is to further generalize these to other situations where one can compare 
deformations of algebras to deformations of schemes. 

The comparison map for projective schemes factors through deformations
of the open subset of $C(X)$ where the vertex $\{0\}$ is
removed. Thereafter one 
compares deformations to $X = (C(X)\setminus \{0\}) /k^\ast$ via the
quotient map. A 
natural question is if this can be generalized to closed subschemes of
toric varieties corresponding to ideals in the Cox ring. It turns out
one can go even further, i.e.\ the quotient need not be by a quasi-torus.

In this paper
we consider schemes $X$ that are good quotients of a quasi-affine
scheme $X^\prime$ by a linearly reductive group $G$. We assume that
$X^\prime \subseteq \Spec S$, where $S$ is a finitely generated
$k$-algebra and that $G$ acts on $S$ inducing the 
action on $X^\prime$. We can then compare $\Def_S^G$ to $\Def_X$ where
$\Def_S^G$ is the functor of invariant deformations of $S$. The
precise definitions of these settings are 
formulated with what we call $G$-quadruples - see Definition
\ref{ferry}.  Given such a situation
we define in Definition \ref{subferrydef} a $G$-subquadruple induced
by a $G$-invariant ideal in $S$. This give us a setting to compare local
Hilbert functors with the deformations functor of an invariant
surjection of $k$-algebras.

Linearly reductive groups have many properties coming from the
Reynolds operator which make it
possible to prove things, e.g.\ taking invariants is exact. Another
reason to work with them is that the functor of invariant deformations is
well defined and has the usual nice properties of a good deformation
theory. This was proven by Rim. 

Our main result on the local Hilbert functor is Theorem
\ref{hilbthm}. The conditions are depth conditions along the
complement of $X^\prime$ in $\Spec S$ and along the locus where the
quotient map fails to be geometric and smooth. We state also
corollaries for subschemes of toric 
varieties and weighted projective space. 

For the abstract deformation
functor $\Def_X$ the results are not so exact due to the presence of
infinitesimal 
automorphisms. It is not clear what the correct assumptions should be
but we found it useful to use results of Altmann regarding rigidity of
$\QQ$-Gorenstein toric singularities as a guide. If $\pi: X^\prime \to X$ is the quotient map set $\CS = \pi_\ast \CO_{X^\prime}$. As in the Hilbert functor case we get conditions involving the depth of $\CS$ 
along the locus in $X$ where the quotient map fails to be geometric and smooth, but also where the isotropy groups are not finite. 

A  new ingredient is what we call a set of 
Euler derivations coming from the Lie algebra $\mathfrak{g}$ of $G$. These are a generator set for the sheaf of derivations of $\CS$ over $\CO_X$ which we prove is free. They define an equivariant map $E:
\Omega^1_{\CS/\CO_X} \to \CS \otimes \mathfrak{g}^\ast$. The cokernel 
$\mathcal{Q}$ plays an important role for the obstructions to comparing  $\Def^G_{S}$ and $\Def_{X}$. The support of $\mathcal{Q}$ is contained in the set of points $x \in X$ where $\pi^{-1}(x)$ fails to have finite isotropy. In particular for toric varieties it is contained in the non-simplicial locus.

Our most general result for $\Def_X$ is Theorem
\ref{defGtheorem}. The results are more easily presented when $G$ is a quasitorus, i.e.\ the product of a torus and a finite abelian group (Theorem 
\ref{cor1}) and even better when $\mathcal{Q} = 0$ (Theorem \ref{Q0thm}). If $S$ is a regular ring then we
get a rigidity statement as corollary.

As a by-product of our generalization of the Euler derivation we are able to give criteria for when there
exists a generalized Euler sequence for the scheme $X$. See Definition
\ref{defeuseq} and Theorem \ref{thmeuseq}.

We conclude with examples of how our results can be used to study
deformations of toric varieties and subschemes of these. In particular
we consider Calabi-Yau hypersurfaces in simplicial toric Fano
varieties, first order deformations of toric singularities and reprove
rigidity results of Altmann (\cite[6.5]{al:min}) and of Totaro (\cite[Theorem 5.1]{to:ju}).

Throughout this paper $k$ is an algebraically closed
  field, in Section \ref{Defsec} we assume characteristic 0 and in
  Section \ref{tvsec}  that $k = \CC$.

\begin{ack} We
would  in particular like to thank Dmitry Timashev for patiently explaining aspects
of invariant theory that led to the the correct setting for our
results. We are grateful to Nathan Owen Ilten, Manfred Lehn, Benjamin
Nill and Arne B. Slets\o e for helpful discussions and answering questions.  
\end{ack}

\section{Preliminaries}

\subsection{Cotangent cohomology}

To fix notation we give a short description of the cotangent modules
and sheaves. Given a ring $R$ and an $R$-algebra $S$ there is a complex of free $S$
modules; the {\it cotangent complex} $\mathbb L_{\bullet}^{S/R}$. See
e.g.\ \cite[p. 34]{an:hom} for a definition.  For an $S$ module $M$ we
get the {\it cotangent cohomology} modules $T^i(S/R;M)
=H^i(\Hom_S(\mathbb L_{\bullet}^{S/R},M))$. If $R$ is the ground field
we abbreviate $T^i(S/R;M)=T^i_S(M)$ and $T^i_S(S)=T^i_S=T^i_X$ if
$X=\Spec S$.

If $X$ is a scheme we may globalise these as follows. If $\mathcal R$
is a sheaf of rings on $X$ and $\mathcal S$ an $\mathcal R$ algebra we set
${\mathcal L}_{\bullet}^{{\mathcal S}/{\mathcal R}}$ to be the complex
of sheaves associated with the presheaves $U\mapsto {\mathbb
L}_{\bullet}^{{\mathcal S}(U)/{\mathcal R(U)}}$. Let $\mathcal F$ be an
${\mathcal S}$ module. We get the cotangent cohomology sheaves
${\mathcal T}^i({\mathcal S}/{\mathcal R};{\mathcal F})$ as the
cohomology sheaves of ${\mathcal Hom}_{\mathcal S}({\mathcal
L}_{\bullet}^{{\mathcal S}/{\mathcal R}},{\mathcal F})$ and the
cotangent cohomology groups $T^i({\mathcal S}/{\mathcal R};{\mathcal
F})$ as the cohomology of $\Hom_{\mathcal S}({\mathcal
L}_{\bullet}^{{\mathcal S}/{\mathcal R}},{\mathcal F})$.

Because of the functoriality of these constructions we have ${\mathcal
T}^i({\mathcal S}/{\mathcal R};{\mathcal F})$ as the sheaf associated
to the presheaf $U\mapsto T^i({\mathcal S}(U)/{\mathcal
R}(U);{\mathcal F}(U))$ and $T^{\bullet} ({\mathcal S}/{\mathcal
R};{\mathcal F})$ as the hyper-cohomology of ${\mathcal Hom}_{\mathcal
S}({\mathcal L}_{\bullet}^{{\mathcal S}/{\mathcal R}},{\mathcal F})$. In
particular there is a ``local-global" spectral sequence
\begin{equation}\label{local-global} H^p(X,{\mathcal T}^q({\mathcal
S}/{\mathcal R};{\mathcal F}))\Rightarrow T^{n}({\mathcal S}/{\mathcal
R};{\mathcal F})\, .
\end{equation} If ${\mathcal S}$ is the structure sheaf ${\mathcal
O}_X$ and $\mathcal{R}$ corresponds to the  ground field, then we
abbreviate as above to 
$T^i_X(\mathcal{F})$.

The properties of the cotangent cohomology we use, e.g.\ the
Zariski-Jacobi sequence and flat base change results, may be found in
 \cite{an:hom}. We include here one result that does not seem to
be well known. Let $Z \subseteq X$ be a closed subscheme, then following Laudal
\cite[3.2.10]{la:for} one may define cotangent cohomology with support
in $Z$ denoted $T_Z^{i} ({\mathcal O}_X/{\mathcal R};{\mathcal F})$.
If $Z \subseteq X=\Spec S$ is given by $V(I)$, we write $T_{I}^i(S/R;S)=
T_{Z}^i(\CO_X/\mathcal{R};\CO_X)$.
\begin{theorem}\emph{\cite[Theorem 3.2.11]{la:for}}\label{supp} There
is a long exact sequence
\begin{equation*}\dots \to T_Z^{i} ({\mathcal O}_X/{\mathcal
R};{\mathcal F}) \to T^{i} ({\mathcal O}_X/{\mathcal R};{\mathcal F})
\to T^{i} ({\mathcal O}_{X \setminus Z}/{\mathcal R};{\mathcal F}) \to
T_Z^{i+1} ({\mathcal O}_X/{\mathcal R};{\mathcal F}) \to
\dots \end{equation*} and a spectral sequence yielding
\begin{equation*}T^{p} ({\mathcal O}_X/{\mathcal
R};\mathcal{H}^q_Z({\mathcal F}))\Rightarrow T_Z^{n} ({\mathcal
O}_X/{\mathcal R};{\mathcal F})\, .
\end{equation*}
\end{theorem}
Here $\mathcal{H}^q_Z({\mathcal F})$ are the local cohomolgy sheaves, see Section \ref{depthcon}.

\vspace{1pt}

\subsection{Deformation functors}

The deformation theory we use in this paper is described in various degrees
of generality and readability in 
\cite{ls:cot}, \cite{sc:fun}, \cite{il:com}, \cite{la:for},
\cite{ser:def} and \cite{ha:def}. For a slightly different,  but
applicable, newer
approach see e.g.\ \cite{fm:obs}.  For the functor of invariant
deformations see \cite{ri:eq}.

Let $\mathbf{C}$ be the category of Artin local $k$-algebras with
residue field $k$. We list here the deformation functors on
$\mathbf{C}$ of interest to us.  
We denote the functor of deformations of a scheme $X$ by $\Def_X$. 
The functor of embedded deformations of a subscheme $X
  \subset Y$ is  denoted $\Hilb_{X/Y}$ and called the local Hilbert
  functor. The deformations of an
  $R$-algebra $S$ is denoted $\Def_{S/R}$ and if $R = k$ we simply
  write $\Def_{S}$. 

If $G$ is an algebraic group acting on a scheme $X$ then it also acts
on the set of deformations over $A$. If $\sigma \in G$ and $f: \mathcal{X} \to \Spec A$ is a deformation then $\sigma$ acts by
\begin{equation*}
\begin{tikzcd} X \arrow[hookrightarrow]{r}{i} \arrow{d} &
  \mathcal{X} \arrow{d}{f} 
\\ \Spec k \arrow[hookrightarrow]{r}& \Spec A 
\end{tikzcd}
\mapsto 
\begin{tikzcd} X \arrow[hookrightarrow]{r}{i \circ \sigma^{-1}} \arrow{d} &
  \mathcal{X} \arrow{d}{f} 
\\ \Spec k \arrow[hookrightarrow]{r}& \Spec A 
\end{tikzcd}
\, .
\end{equation*} 
If $G$ is linearly reductive then there is a well behaved
sub-deformation functor of 
invariant deformations $\Def_X^G$. Similarly
if $R \to S$ is equivariant for a linearly
  reductive group $G$ then there is a subfunctor  $\Def_{S/R}^G$ of
  invariant deformations. If $G$ is a quasi-torus so that the action
  corresponds to a grading by an abelian group $C$ then we often write
  $\Def_{S/R}^0$ for the degree $0 \in C$ deformations instead of
  $\Def_{S/R}^G$. 

The tangent and obstruction spaces for $\Def_X$ are $T^i_X =
T^i({\CO_X}/k;\CO_X)$ for $i=1$ and $2$. If $f: X \to Y$ is a closed
embedding then the tangent and obstruction spaces for $\Hilb_{X/Y}$ are $
T^i({\CO_X}/f^{-1} \CO_Y;\CO_X)$ for $i=1$ and $2$. The tangent and
obstruction spaces for $\Def_{S/R}$ are $ T^i(S/R;S)$ for $i=1$ and
$2$ and finally the tangent and
obstruction spaces for $\Def_{S/R}^G$ are $ T^i(S/R;S)^G$ for $i=1$ and
$2$. If $D$ is one of these deformation functors let $T^i_D$, $i=1,2$,
denote the corresponding tangent and obstruction space.

Recall that a morphism $F \to G$ of functors is \emph{smooth} if for any
surjection $B \to A$ in $\mathbf{C}$, the morphism
$$F(B) \to F(A) \times_{G(A)} G(B)$$
is surjective.
A functorial map of deformation
functors $D \to D^\prime$ induces maps $T^i_D \to T^i_{D^\prime}$. We
use throughout the standard criteria for smoothness, namely that
$T^1_D \to T^1_{D^\prime}$  is surjective and $T^2_D \to
T^2_{D^\prime}$ is 
injective. See \cite[Proposition 2.3.6]{ser:def} for a more general statement and proof.

If $T^1_D \to T^1_{D^\prime}$  is an isomorphism and $T^2_D \to
T^2_{D^\prime}$ is injective it is not necessarily true that $D \to
D^\prime$ is an isomorphism. This is true  if $D$ and $D^\prime$
satisfy Schlessinger's condition $\text{H}_4$, see
\cite[2.11 and 2.15]{sc:fun}. If $X \subset Y$ is a closed subscheme
then $\Hilb_{X/Y}$ satisfies $\text{H}_4$. In general $\Def_X$ does
not, but in our case  we will not only have an isomorphism at the
tangent level but also surjectivity of infinitesimal  automorphisms
allowing us to state that the functors we compare are isomorphic
(Lemma \ref{defGlemma}).

\vspace{1pt}

\section{$G$-quadruples}
\subsection{Definitions} We consider now a standard situation in invariant
theory. Definitions of different types of quotients in algebraic geometry vary slightly in the literature. For us the best one is the original notion of good (and geometric) quotient due to Seshadri. 
\begin{definition}[\cite{ses:quo} Definition 1.5] Let $G$ be an affine algebraic group acting on an algebraic scheme $Y$ and $\pi: Y \to X$ a morphism. Then we say that $\pi$ is a \emph{good quotient} if the following properties hold:
\begin{list}{\textup{(\roman{temp})}}{\usecounter{temp}}
\item $\pi$ is a surjective, affine $G$-invariant morphism
\item ($\pi_\ast \CO_Y)^G = \CO_Y$
\item if $W \subseteq Y$ is closed and $G$-invariant then $\pi(W)$ is closed in $X$
\item if $W_1$ and $W_2$ are closed, disjoint and $G$-invariant in $Y$, then $\pi(W_1)$ and $\pi(W_2)$ are disjoint in $X$.
\end{list} 
\end{definition}
In this case one writes $X=Y//G$. If $G$ is reductive acting algebraically on $Y=\Spec(A)$ then $Y \to \Spec(A^G)$ is a good quotient. (See e.g.\ \cite[Theorem 1.1]{ses:quo}.)
If $\pi:Y \to Y//G$ is a good quotient and all $G$-orbits are closed then $\pi$ is called a \emph{geometric} quotient (\cite[Definition 1.6]{ses:quo}). The condition closed orbits is equivalent to that $\pi$ induces a bijection between $G$-orbits and $Y//G$.
\begin{remark} It will be essential for us that the good quotient map is affine. In some definitions of good this is not included. That would allow e.g.\ the non-separated quotient $(\mathbb{C}^2 \setminus \{(0,0\})//C^\ast$ where $C^\ast$ acts by $\diag(\lambda, \lambda^{-1}).$\end{remark}

 It is convenient to give the situation we study a name.
\begin{definition} \label{ferry} Let $G$ be a linearly reductive 
algebraic group. Let
$X$ be a noetherian $k$-scheme, 
$Z \subsetneq X$ a (possibly empty) closed subset, 
$S$ a finitely generated $k$-algebra on which $G$ acts and 
$J \subseteq S$ an ideal such that the $G$ action restricts to an action on $\Spec S \setminus V(J)$.  We call $(X,Z,S,J)$ a \emph{$G$-quadruple} if the following properties are satisfied:
If $X^\prime =
\Spec S \setminus V(J)$ and $U = X \setminus Z$ then there is a commutative
diagram
\begin{equation*}
\begin{tikzcd} U^\prime \arrow[hookrightarrow]{r} \arrow{d}{\pi_{\mid
U^\prime}} &X^\prime \arrow[hookrightarrow]{r} \arrow{d}{\pi} &\Spec S
\\ U \arrow[hookrightarrow]{r}&X &
\end{tikzcd}
\end{equation*} where $U^\prime = \pi^{-1}(U)$ and 
\begin{list}{\textup{(\roman{temp})}}{\usecounter{temp}}
\item $\pi$ is a good quotient of $X^\prime$ by the action $G$
\item $\pi_{\mid U^\prime}$ is a  geometric quotient and smooth of
  relative dimension $\dim G$
\item $\depth_J S \ge 1$. 
\end{list} 
The sheaf of algebras $\CS = \pi_\ast \CO_{X^\prime}$ is called
\emph{the associated sheaf of algebras} of the $G$-quadruple.
\end{definition}

We will throughout the rest of this paper use the above
notation to refer to the various objects in the definition. 

\begin{remark} A $G$-quadruple may be constructed using Geometric Invariant
Theory. (See \cite{mfk:geo},  although the presentation in
\cite{mu:int} applies more directly to our situation.) Starting with
the action of $G$ on a reduced and irreducible affine scheme $Y = \Spec S$ and a character
$\chi$ of $G$, let $Y^{ss}$ and $Y^s$ be the semistable and stable points of
$Y$ with respect to $\chi$. This gives a $G$-quadruple by letting
$X^\prime = Y^{ss}$, $J = I(Y \setminus Y^{ss})$, $X = Y^{ss}//G$, 
$U^\prime$ equals the open subset of $Y^s$ where the quotient map is
smooth, $U = U^\prime/G$ and $Z = X \setminus U$. 
\end{remark}

\subsection{Examples} Here are some examples.

\begin{example} \emph{The $\Spec$ construction.} If $X = \Spec S$ and
  $G$ is trivial then $(X,\emptyset,S,(1))$ is a $G$-quadruple.
\end{example}

\begin{example} \emph{The usual $\Proj$ construction.} If $X = \Proj
S$ for  $S$ finitely generated $\ZZ$-graded $k$-algebra
generated in degree 1 with irrelevant ideal $\mathfrak{m}$ and $G=
k^\ast$ then $(X,\emptyset,S,\mathfrak{m})$ is a $G$-quadruple.
\end{example}

\begin{example} \label{wps} \emph{The $\Proj$ construction in general.} 
Let $S$ be a finitely generated $\ZZ_{+}$ graded $k$-algebra with $S_0
= k$ and set $X = \Proj S$. Assume $S =R/I$ where $R = k[x_0, \dots ,
x_n]$ is graded with $\deg x_i = 
q_i \in \N$. Let $\mathfrak{m} = (x_0, \dots , x_n)$ and suppose
$\depth_\mathfrak{m} S \ge 1$. This defines an embedding of $X$ into
the weighted projective space $\PP(\mathbf{q}) = \PP(q_0, \dots ,
q_n)$.

Following Miles Reid we say that $\PPq$ is \emph{well formed} if no
$n$ of the $q_0, \dots , q_n$ have a common factor. Similarly we
will say that the corresponding grading on $S$ is well formed. For every $\mathbf{q} \in \N^{n+1}$ there exists a well formed grading $\bf q'$ such that $\PP(\bf q') \simeq \PP(\bf q)$, see
e.g.\ \cite[Proposition 1.3]{de:esp} or \cite[1.3.1]{do:wei}.
Let $J^\prime_k$ be the ideal of $R$ generated by $\{x_i : k \nmid q_i\}$ and
set $$J^\prime = \bigcap_{k \ge 2} J^\prime_k = \bigcap_{\substack{p \text{ prime}
\\ p \mid \lcm(q_0, \dots , q_n)}} J^\prime_p \, .$$ The singular locus
of the well formed $\PPq$ is $Z = V(J^\prime) \subseteq \PPq$ and satisfies $\codim_Z \PPq \ge 2$.
 
Then
$(X,X \cap Z,S,\mathfrak{m})$ is a $G$-quadruple.
\end{example}

\begin{example} \label{TVex} \emph{The Cox construction for complex toric
varieties.} This is our main example and it includes the previous
ones, see e.g.\ \cite{co:hom} and \cite[Chapter 5]{cls:tor}. To fix
notation for the rest of this paper we recall the construction. Using the
standard notation for toric geometry let $X = X_\Sigma$ be an
$n$-dimensional toric variety
given by a fan $\Sigma$ in $N_\RR$. We assume here for simplicity that
$X_\Sigma$ has no torus factors, but this is not necessary for
applying our results (see \cite[5.1.11]{cls:tor}). Let $\Sigma(1) =
\{\rho_1, \dots , \rho_N\}$ be the set of rays of $\Sigma$ and let
$v_i$ be the primitive generator of $\rho_i \cap N$. The
divisor class group of $X$ is given by the exact sequence 
$$0 \to \ZZ^n \xrightarrow{b} \ZZ^N \to \Cl(X) \to 0$$
where $b(u) = (\langle u, v_1\rangle  , \dots , \langle u,
v_N\rangle )$. The ring $S = \CC[x_\rho
: \rho \in \Sigma(1)]$  is naturally graded by the abelian group 
$\Cl(X)$ and with this grading it is called the \emph{Cox ring} or
\emph{total homogeneous coordinate ring}  of $X$. 

For each cone $\sigma$ in $\Sigma$ there is a monomial 
$$\prod_{\rho_i \nsubseteq \sigma} x_i \in S$$
and define $B(\Sigma)$ to be the ideal of $S$ generated by these. It
is called the \emph{irrelevant ideal} of the Cox ring. Let $Z(\Sigma)
= V(B(\Sigma)) \subseteq \Spec S$. Let $G$ be the
quasi-torus $\Hom_{\ZZ}(\Cl(X_\Sigma), \CC^\ast)$. The theorem of Cox
states that $X$ is an almost geometric quotient for the action of $G$
on $\Spec S \setminus Z(\Sigma)$. 

Let $\Sing(X_\Sigma)$ be the singular locus, i.e.\
$$\Sing(X_\Sigma) = \bigcup_{\substack{\sigma \in \Sigma
\\ \sigma \text{ not smooth}}} V(\sigma)$$ where $V(\sigma)$ is the
closure of the torus orbit corresponding to $\sigma$. The smooth locus is
given by the subfan consisting of smooth cones in $\Sigma$. Then it follows 
from the construction (see e.g.\ \cite[Exercise 5.1.10]{cls:tor}) that
$(X_\Sigma,\Sing(X_\Sigma),S,B(\Sigma))$ is a $G$-quadruple.
Note that a $G$ invariant $S$-module $M$ is the same thing as a $\Cl(X)$
graded $S$-module and that $M^G = M_0$, the degree $0 \in \Cl(X)$ part
of $M$. Moreover if $\alpha = [D] \in \Cl(X)$ then $S_\alpha =
\Gamma(X,\CO_X(D))$. In particular the associated sheaf of algebras of the $G$-quadruple is in this case 
$$\CS = \bigoplus_{[D] \in \Cl(X)} \CO_X(D)\, .$$
\end{example}

\begin{example} \label{qs} \emph{Quotient singularities.} Let $G
\subset \GL_n(\CC)$ be a finite group without pseudo-reflections. Then
$(\CC^n/G,\Sing(\CC^n/G),\CC[x_1, \dots , x_n],(1))$ is a $G$-quadruple.
\end{example}

\begin{example} \label{grass} \emph{Grassmannians.} Consider the Grassmannian
  $\mathbb{G}(d,n)$ and let $S$ be the polynomial ring on variables
  $x_{ij}$ for $i= 1, \dots , d$ and $j = 1, \dots , n$. Thus $G =
  \GL_d$ acts on $S$ by viewing the variables as entries in a $d
  \times n$ matrix. Let $J$ be the ideal generated by the
  $\binom{n}{d}$ maximal minors in such a matrix. Then
  $(\mathbb{G}(d,n), \emptyset, S, J)$ is a $G$-quadruple.
\end{example}

\begin{example} \emph{Moduli spaces.}  Compactifications of moduli spaces
constructed using GIT on affine schemes give $G$-quadruples as
explained above. Among these are the moduli of smooth
hypersurfaces in $\PP^n$, vector bundles and quiver representations. 
\end{example}

\subsection{ $G$-subquadruples} We may also define $G$-subquadruples. 

\begin{definition} \label{subferrydef} Given a $G$-quadruple
  $(Y,W,R,J)$ and a $G$-invariant ideal $I \subseteq R$ 
we may construct a new $G$-quadruple as
follows. Let $S = R/I$,  set $\bar{J} = (I +  J)/I$ and assume
$\depth_{\bar{J}} S \ge 1$. Let $\pi : Y^\prime \to Y$ be the good quotient as in
Definition \ref{ferry}. The quasi-affine scheme $X^\prime = \Spec S
\setminus V(\bar{J})$ is a closed $G$-invariant subset of $Y^\prime$.  Thus
$\pi: X^\prime \to X = \pi(X^\prime)$ is a good quotient where $X$ has
structure sheaf $(\pi_\ast \CO_{X^\prime})^G$. It follows that $(X, X
\cap W, S, \bar{J})$ is also a $G$-quadruple. We call it the
\emph{$G$-subquadruple induced by $I$}.
\end{definition}

A $G$-subquadruple determines a diagram as
in Figure \ref{subferryfig} where $Z = X \cap W$, $U_X = X \cap
U_Y$ and by definition 
$U^\prime_X = \pi^{-1}(U_X)$. Let $\CS = \pi_\ast \CO_{X^\prime}$ be
the associated sheaf of algebras of $(X, Z, S, \bar{J})$. Let $f: X \to Y$ be the closed
embedding. Let
$I_W$ be the radical ideal of $\Spec R \setminus U_Y^\prime$ and set
$I_Z = (I_W + I)/I$. We will use this notation and the notation in the
diagram throughout.

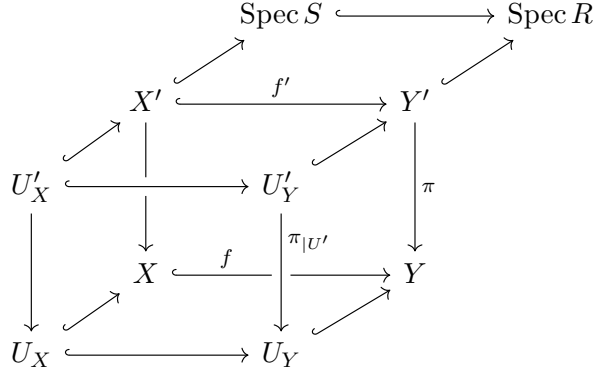
\begin{figure} \centering
\begin{tikzcd}[row sep=scriptsize, column sep=scriptsize] && \Spec S
\arrow[ hookrightarrow]{rr}& & \Spec R\\ & X^\prime
\arrow[hookrightarrow]{rr}{f^\prime}
\arrow{dd}\arrow[hookrightarrow]{ur} & &Y^\prime \arrow{dd}{\pi}
\arrow[hookrightarrow]{ur} \\ U^\prime_X \arrow[crossing over,
hookrightarrow]{rr}\arrow{dd} \arrow[hookrightarrow]{ur} & & U_Y^\prime
\arrow[hookrightarrow]{ur}\\ & X \arrow[hookrightarrow]{rr}[near
start]{f} & & Y \\ U_X \arrow[hookrightarrow]{rr} \arrow[hookrightarrow]{ur} & & U_Y
\arrow[crossing over, leftarrow]{uu}[near end, swap]{\pi_{\mid
U^\prime}} \arrow[hookrightarrow]{ur} \\
\end{tikzcd}
\caption{The diagram for $G$-subquadruples.}
\label{subferryfig}
\end{figure}

Now $\pi_{\mid U^\prime_Y}$ is a
geometric quotient, so in particular inverse images of points are
$G$-orbits. Thus for any $G$-invariant subset $V \subseteq U^\prime_Y$
we have $\pi^{-1}(\pi(V)) = V$. Thus $X^\prime \cap U^\prime_Y =
\pi^{-1}(\pi(X^\prime \cap U^\prime_Y)) = \pi^{-1}(\pi(X^\prime) \cap
\pi(U^\prime_Y))$ since $x, y \in U^\prime_Y$ and $\pi(x) =
\pi(y)$ implies $x$ and $y$ are in the same $G$-orbit. This shows that
$U^\prime_X = X^\prime \cap U^\prime_Y = \Spec S \setminus V(I_Z)$, so
the above diagram is commutative. Note that
the  vertical square involving $U^\prime_X,U^\prime_Y, U_X,U_Y$ is Cartesian, but the one involving $X^\prime, Y^\prime, X,Y$ need not be. 

\subsection{Depth and local cohomology} \label{depthcon} We recall the connection between depth and local cohomology for rings and
schemes which we use through out. See \cite[Section 3]{gr:loc} and
\cite[Exp.\! 1-3]{gr:sga2} for proofs and details. If $S$ is a Noetherian
ring, $I$ an ideal of $S$, and $M$ an $S$-module (not necessarily finitely
generated) we define $\depth_I(M )$ by
\begin{equation*}
 \depth_I(M)=\max\{j \in \ZZ \cup \{\infty\} : H_I^{j-1}(M) = 0\} \ .
\end{equation*}
where $ H_I^{i}(-)$ is the right derived functor of $ \Gamma_I(-)$ and
$ \Gamma_I(M) = \ker(M \to \Gamma(\Spec S \setminus V(I),M))$. By Proposition
2.4 of \cite[Exp.\! 3]{gr:sga2}, if reg denotes the length of a maximal
$M$-regular sequence in $I$ (letting reg = $\infty$ if $IM=M$), then
$H_I^j(M) = 0$ for every $j < \rm{reg}$, and if $M$ is finitely generated then
reg $ = \depth_I(M)$. 

For a quasi-coherent sheaf $\mathcal{F}$ on a noetherian
scheme $X$ we  have the local cohomology groups  $H^i_{Z}(X, \mathcal{F})$ and the local cohomology sheaves $\mathcal{H}_Z^i(\mathcal{F})$ with the spectral sequence
\begin{equation*}H^p(X,\mathcal{H}_Z^q(\mathcal{F}))\Rightarrow H^n_{Z}(X, \mathcal{F})\, .
\end{equation*}
We define depth as in
the ring case by
\begin{equation*}
 \depth_Z(\mathcal{F}) = \inf \{i \in \ZZ \cup \{\infty\} :  \mathcal{H}_Z^i(\mathcal{F}) \ne 0 \}
\end{equation*}
 (see  \cite[Proposition 2.2]{gr:loc} and compare with \cite[Theorem 3.8]{gr:loc}). 
 Note that if $\mathcal{F}$ is coherent, we have  $\depth_Z \mathcal{F} = \inf_{x \in Z} \depth \mathcal{F}_x$  where $\depth \mathcal{F}_x$ is the depth of
$\mathcal{F}_x$ as $\CO_{X,x}$ module (see  \cite[Corollary 3.6]{gr:loc}).
 
 If $U = X \setminus Z$ and $j : U \to X$ is the inclusion then there are exact sequences 
 \begin{equation*} \cdots \to H^i_Z(X, \mathcal{F}) \to H^i(X , \mathcal{F}) \to H^i(U, \mathcal{F}) \to H^{i+1}_Z(X, \mathcal{F}) \to \cdots
 \end{equation*} 
 and
  \begin{align*} 0 \to \mathcal{H}^0_Z(\mathcal{F}) \to \mathcal{F} & \to  j_{\ast}\mathcal{F}_{\mid {U}} \to \mathcal{H}^1_Z(\mathcal{F}) \to 0 \\
  \mathcal{H}^{i+1}_Z(\mathcal{F}) & \simeq R^ij_\ast \mathcal{F}_{\mid {U}} \quad \text{for $i > 0$} \, .
\end{align*}
The condition $ \depth_Z(\mathcal{F}) \ge 2$ will appear many times and we see from the above that it is equivalent to that the natural map $\mathcal{F} \to j_{\ast}\mathcal{F}_{\mid {U}}$ is an isomorphism.

Let $(X, Z, S, \bar{J})$ be a $G$-quadruple with associated sheaf of algebras $\CS$, for instance a
$G$-subquadruple of $(Y,W,R,J)$ induced by $I \subseteq R$, let
$ Z^\prime = V(I_Z)\cap X^\prime $ and note that
$X^\prime = \Spec S \setminus V(\bar{J})$ and
$U^\prime_X = X^\prime \setminus Z^\prime $ are open sets of $ \Spec S$. Since
$\depth_Z \CS$ is so often used in this paper and $\CS$ is in general not 
coherent, we want to relate it to
$\depth_{Z^\prime} \CO_{X^\prime}.$ Note that we have
\begin{equation*}
 H^i(X, \CS) \simeq  H^i(X^\prime, \CO_{X^\prime}) \quad \text{and} \quad  H^{i}(U_X,
 \CS_{\mid U_X}) \simeq H^i(X^\prime \setminus Z^\prime, \CO_{X^\prime})
\end{equation*} for $i \ge 0$ 
because $\pi$ is affine. Moreover using \cite[Corollary 5.6]{gr:loc} and $\pi$
affine we get that
$$H^i_{Z \cap V}(V, \pi_\ast\CO_{X^\prime}) \simeq H^i_{Z^\prime \cap
  V^\prime}(V^\prime, \CO_{X^\prime})$$
for every open affine $V \subseteq X$ and $V^\prime = \pi^{-1}(V)$ open affine
of $\Spec S$. It follows that
$\mathcal{H}^i_Z(\CS)= \pi_\ast(\mathcal{H}^i_{Z^\prime}(\CO_{X^\prime}))$  and
$$\depth_Z \CS= \depth_{Z^\prime} \CO_{X^\prime} \, .$$ 
Notice that the latter is defined by a coherent sheaf, thus characterized by
the length of maximal regular sequences. We also have
\begin{equation}\label{lociso}\mathcal{H}^i_{Z^\prime}(\CO_{X^\prime}) \simeq \widetilde{
  H^i_{Z^\prime}(X^\prime, \CO_{X^\prime})}_{\mid {X^\prime}} \simeq
\widetilde{H^{i}_{I_Z}(S)}_{\mid {X^\prime}}\end{equation}
because $\widetilde{H^{i}_{\bar{J}}(S)}_{\mid {X^\prime}}=0$. More generally
since $ \Spec S $ is affine there is a diagram %with horizontal isomorphisms
\begin{equation*}
  \begin{tikzcd}  &  H^i(X^\prime, \CO_{X^\prime}) \arrow[r, "\sim"]\arrow{d} &
    H^{i+1}_{\bar{J}}(S) \\ 
    &  H^i(X^\prime \setminus Z^\prime, \CO_{X^\prime}) \arrow[r, "\sim"] & H^{i+1}_{I_Z}(S)
\end{tikzcd}
\end{equation*} for $i > 0$; for $i=0$ the horisontal maps are surjective with
kernels $S$ and $\coker(H^0_{I_Z}(S) \to S)$ respectively. Since the vertical
map fits into 
a long exact sequence involving the local cohomology group
$H^i_{Z^\prime}(X^\prime, \CO_{X^\prime})$ we get an exact sequence of $S$-modules
\begin{equation} \label{loexse} \cdots \longrightarrow H^i_{Z^\prime}(X^\prime,
  \CO_{X^\prime}) \longrightarrow H^{i+1}_{\bar{J}}(S) \longrightarrow
  H^{i+1}_{I_Z}(S) \longrightarrow H^{i+1}_{Z^\prime}(X^\prime,
  \CO_{X^\prime}) \longrightarrow \cdots\ .
\end{equation}

  \begin{remark}
We may relate $\depth_{I_Z} S$, $\depth_{\bar{J}} S$ and
  $\depth_Z \CS = \depth_{Z^\prime} \CO_{X^\prime}$. We get
$$\depth_Z \CS \ge d \quad \text{and} \quad \depth_{\bar{J}} 
  S \ge d \, \Longleftrightarrow \, \depth_{I_Z} S \ge d \, .$$
  Indeed, the implication $\Leftarrow$ follows from $\depth_{I_Z} S \le  \depth_{\bar{J}} 
  S$ (due to $I_Z \subseteq \bar{J}$) and \eqref{lociso}. The  implication $\Rightarrow$ follows from \eqref{loexse} and the spectral sequence above.
\end{remark}

\section{Deformations of the embedded scheme - $\Hilb_{X/Y}$}

We begin with a general lemma.
\begin{lemma}\label{sit1} If in a Cartesian square of schemes
\begin{equation*}
\begin{CD} X^\prime @>f^\prime>> Y^\prime \\ @VVV @VV\pi V\\ X @>f>> Y
\end{CD}
\end{equation*} the morphism $\pi$ is flat and affine and $f$ is a
closed immersion then 
$$T^{i} ({\mathcal O}_X/ f^{-1}\mathcal{O}_Y;\pi_\ast\CO_{X^\prime})
\simeq T^{i} ({\mathcal
O}_{X^\prime}/{f^\prime}^{-1}\mathcal{O}_{Y^\prime};\CO_{X^\prime})$$
for all $i \ge 0$.
\end{lemma}
\begin{proof} We first compare the corresponding sheaves. The diagram is Cartesian so $f^\prime$ is also a closed immersion. Closed immersions are affine and the base change of an affine morphism is affine, so all
morphisms in the diagram are affine. 

Let $V = \Spec B \subseteq Y$ be an open affine subset, $U = \Spec A = f^{-1}(V)$, $\pi^{-1}(V) = \Spec B^\prime$ and $\pi^{-1}(U) = \Spec A^\prime$.  The diagram locally corresponds to a cocartesian diagram of rings
\begin{equation*}
\begin{CD} A^\prime @<<< B^\prime \\ 
@AAA @AA\pi^{\#}  A \\ A
  @<<< B 
\end{CD} 
\end{equation*}
 with $\pi^{\#} $ flat, so $T^q(A/B; A^\prime) \simeq
T^q(A^\prime/B^\prime; A^\prime)$ for all $q \ge 0$ (\cite[Appendice, Proposition 76]{an:hom}). Applying the corresponding cotangent cohomology sheaves we have \begin{align*}\mathcal{T}^{q} ({\mathcal O}_X/
f^{-1}\mathcal{O}_Y;\pi_\ast\CO_{X^\prime})(U) &= T^q(A/B; A^\prime) \\
\mathcal{T}^{q} ({\mathcal
O}_{X^\prime}/{f^\prime}^{-1}\mathcal{O}_{Y^\prime};\CO_{X^\prime})(\pi^{-1}(U)) &= T^q(A^\prime/B^\prime; A^\prime)\end{align*}
thus
$$\mathcal{T}^{q} ({\mathcal O}_X/
f^{-1}\mathcal{O}_Y;\pi_\ast\CO_{X^\prime}) \simeq \pi_\ast
\mathcal{T}^{q} ({\mathcal
O}_{X^\prime}/{f^\prime}^{-1}\mathcal{O}_{Y^\prime};\CO_{X^\prime})$$
for all $q \ge 0$. 

Again because the maps are affine, $H^p(X, \pi_\ast \mathcal{F}) \simeq H^p(X^\prime,
\mathcal{F})$ for any quasi-coherent $\mathcal{F}$. Thus all terms in
the two spectral sequences \eqref{local-global} for the two cohomology
groups are isomorphic and the result follows.
\end{proof}

\begin{lemma} \label{klep} Let $(X, Z, S, \bar{J})$ be a $G$-subquadruple of
  $(Y,W,R,J)$ induced by $I \subseteq R$ with associated sheaf of algebras
  $\CS$. If $\depth_Z \CS \ge 1$ and the natural map
  $H^0(X, \CS) \rightarrow H^0(U_X, \CS_{\mid U_X})$ is surjective, equivalently
  $H^0(X, \CS) \simeq H^0(U_X, \CS_{\mid U_X})$, then
\begin{list}{\textup{(\roman{temp})}}{\usecounter{temp}}
\item $\depth_Z \CO_X \ge 1$
\item $H^0_{\bar{J}}(S) = H^0_{I_Z}(S) = 0$
\item there is an isomorphism $H^1_{\bar{J}}(S) \simeq H^1_{I_Z}(S)$.
\end{list} 
\end{lemma}
\begin{proof} Suppose $\depth_Z\CS \ge 1$ and   $H^0(X, \CS) \rightarrow H^0(U_X, \CS_{\mid U_X})$ is surjective.
Since taking invariants is exact (i) follows
  from $\mathcal{H}^i_Z(\CS)^G = \mathcal{H}^i_Z(\CO_X)$. We have $\depth_{\bar{J}} S \ge 1$ and
  $I_Z \subseteq \bar{J}$ by the assumption that $(X,Z,S,\bar{J})$ is a
  $G$-subquadruple, so we get (ii) and (iii) by considering the commutative
  diagram
\begin{equation*}
\begin{tikzcd} & 0 \arrow{r} & S \arrow{r} \arrow[equals]{d} & H^0(X,
  \CS) \arrow{r} \arrow{d} &
  H^1_{\bar{J}}(S)  \arrow{r} & 0\\ 
0 \arrow{r} & H^0_{I_Z}(S) \arrow{r} & S \arrow{r} & H^0(U_X,
\CS_{\mid U_X}) \arrow{r} & H^1_{I_Z}(S)  \arrow{r} & 0
\end{tikzcd}
\end{equation*}
where the middle vertical map is an isomorphism by $H^0_{Z}(X, \mathcal{S})
\simeq 
H^0(X,\mathcal{H}_Z^0(\mathcal{S})) = 0$ and 
assumption. 

Finally, to show that $\depth_Z \CS \ge 1$ and 
  $H^0(X, \CS) \rightarrow H^0(U_X, \CS_{\mid U_X})$ surjective is equivalent to
  $H^0(X, \CS) \simeq H^0(U_X, \CS_{\mid U_X})$, suppose the vertical map is an isomorphism. Then the diagram above (or the remark in Section \ref{depthcon}) implies $H^0_{I_Z}(S)=0$, hence $\depth_Z\CS \ge 1$ by \eqref{lociso}.
\end{proof}

\begin{theorem} \label{hilbthm} If $(X, Z, S, \bar{J})$ is a $G$-subquadruple
  of $(Y,W,R,J)$ induced by $I \subseteq R$ with associated sheaf of algebras
  $\CS$  and
\begin{list}{\textup{(\roman{temp})}}{\usecounter{temp}}
\item %$\depth_Z\CO_X \ge 2$ and 
$H^0(X, \CS) \simeq H^0(U_X, \CS_{\mid U_X})$ and \ $\mathcal{H}^1_Z(\CO_X)=0$
\rm(e.g. $\depth_Z \CS \ge 2$\rm),
\item $\Hom_R(I, H^1_{\bar{J}}(S))^G = 0$
\end{list} 
then $\Def^G_{S/R}$ and $\Hilb_{X/Y}$ are isomorphic deformation functors.
\end{theorem}
\begin{proof} To see that  $\depth_Z \CS \ge 2$ implies that the assumptions in (i) hold use
  $\mathcal{H}^i_Z(\CS)^G = \mathcal{H}^i_Z(\CO_X)$, the spectral sequence of Section \ref{depthcon} which implies $H^i_Z(X,\CS) = 0$ for $i \le 1$, and the diagram in the
  proof of Lemma~\ref{klep}.

We will prove that both deformation functors are isomorphic
  to $\Hilb_{U_X/U_Y}$. Note first that if $f: X \to Y$ is any closed
  embedding then $T^0(\CO_X/f^{-1}\CO_Y; \mathcal{F}) = 0$ for any
  $\CO_X$-module $\mathcal{F}$. Thus the spectral sequence in Theorem
  \ref{supp} yields $T_{I_Z}^{1} (S/R;S) \simeq \Hom_R(I, H^0_{I_Z}(S)) = 0$
  by Lemma \ref{klep}. Furthermore this and Lemma \ref{klep} (iii) show that
  $T_{I_Z}^{2} (S/R;S) \simeq \Hom_R(I, H^1_{I_Z}(S)) \simeq \Hom_R(I,
  H^1_{\bar{J}}(S))$.

We start with applying the long exact sequence in
  Theorem \ref{supp} and get
\begin{equation*}0 \to T^{1} (S/R;S)
\to T^{1} ({\mathcal O}_{U^\prime_X}/\mathcal{R};{\mathcal O}_{U^\prime_X}) \to
T_{I_Z}^{2} (S/R;S)  \to T^{2} (S/R;S)
\to T^{2} ({\mathcal O}_{U^\prime_X}/\mathcal{R};{\mathcal
  O}_{U^\prime_X}) \end{equation*} 
where we set $i_Y: U^\prime_Y \hookrightarrow \Spec R$ and  $\mathcal{R} = (f^\prime)^{-1} i_Y^{-1} \CO_{\Spec R}$. Now $i_Y$ is an open immersion so $\mathcal{R}  \simeq (f^\prime)^{-1} \CO_{U^\prime_Y}$. By assumption 
$\pi_{\mid U^\prime_Y}$ is smooth, in particular flat so 
 by Lemma \ref{sit1}
$$ T^{i}({\mathcal
  O}_{U^\prime_X}/\mathcal{R};{\mathcal
  O}_{U^\prime_X}) \simeq T^{i}({\mathcal
  O}_{U_X}/f^{-1} \CO_{U_Y};\pi_{\ast}{\mathcal
  O}_{U^\prime_X}) \, .$$
 Therefore after taking invariants and using the conditions we see
 that 
$$T^i(S/R; S)^G \to  T^{i}({\mathcal
  O}_{U_X}/f^{-1} \CO_{U_Y};{\mathcal
  O}_{U_X}) $$
is an isomorphism for $i=1$ and injective for $i=2$. Thus
$\Def_{S/R}^G \simeq \Hilb_{U_X/U_Y}$.

We have $\mathcal{H}^i_Z(\CO_X) = 0$ for $i \le 1$ by Lemma \ref{klep} and
assumption, so again the spectral sequence and vanishing of
$T^0(\CO_X/f^{-1} \CO_Y; \CO_X)$ implies
$T^i_Z(\CO_X/f^{-1} \CO_Y; \CO_X) = 0$ for $i \le 2$. We apply the long exact
sequence in Theorem \ref{supp} again to get
$T^{1} ({\mathcal O}_X/f^{-1}\CO_Y;\CO_X) \simeq T^{1} ({\mathcal
  O}_{U_X}/f^{-1}\CO_{U_Y};\CO_{U_X})$
and
$T^{2} ({\mathcal O}_X/f^{-1}\CO_Y;\CO_X) \to T^{2} ({\mathcal
  O}_{U_X}/f^{-1}\CO_{U_Y};\CO_{U_X})$
injective. Thus $\Hilb_{X/Y} \simeq \Hilb_{U_X/U_Y}$.
\end{proof}

\begin{remark}
(i) From the proof we see that the slightly weaker assumptions
$\depth_Z \CS \ge 1$, $\Hom_R(I, H^1_{I_Z}(S))^G = 0$ and
$\Hom_{\CO_Y}(\mathcal{I}, \mathcal{H}^1_Z(\CO_X)) = 0$ also imply the
result in the theorem.

(ii) If the theorem applies and $T^2(S/R;S)^G = 0$ then $\Hilb_{X/Y}$ is
unobstructed even though $H^1(X,\mathcal{N}_{X/Y})$ or $H^2(X,
\mathcal{T}^2_{X/Y})$ do not vanish. See Example \ref{cex}.
\end{remark}

\begin{corollary} \label{} If $(X, Z, S, \bar{J})$ is a $G$-subquadruple of
  $(Y,W,R,J)$ induced by $I \subseteq R$ and $\depth_{I_Z} S \ge 2$,
   then $\Def^G_{S/R}$ and $\Hilb_{X/Y}$ are isomorphic deformation functors.
  %deformation functors.
\end{corollary}
\begin{proof} Both assumptions in Theorem~\ref{hilbthm} are satisfied because
  $\depth_{I_Z} S \ge 2$ implies $\depth_{\bar J} S \ge 2$, whence
  $ H^1_{\bar{J}}(S)=0$ and $\depth_Z \CS \ge 2$ by remark of subsection
\ref{depthcon}.
\end{proof}
In the case of the Cox construction for toric varieties (Example \ref{TVex})
we get a corollary which is a generalization of the comparison theorem as
stated in \cite{ps:hil}, cf.\ \cite[Theorem 3.6 and Remark 3.7]{kle:def}.
Actually \cite[Remark 3.7]{kle:def} implies the comparison
theorem in \cite{ps:hil}. Proofs and full statements of the results we use
here may be found in \cite[Chapter 5 and 6]{cls:tor}.

Let $Y$ be a toric variety with Cox ring $R$. Every
closed subscheme $X$ of $Y$ corresponds to a homogeneous, with respect
to the $\Cl(Y)$ grading, ideal $I \subseteq R$, \cite[Proposition 6.A.6]{cls:tor}. Moreover there is a sheafification
construction taking any graded $R$-module $M$ to a sheaf
$\widetilde{M}$ on $Y$. In particular if $S = R/I$ then $\widetilde{S}
= \bigoplus_{\alpha \in \Cl(Y)} \CO_X(\alpha)$. Also, as in the case
of projective space, one may compute sheaf cohomology from local
cohomology at the irrelevant ideal (\cite[Proposition
2.3]{ems:lo}). In particular there is an exact sequence
\begin{equation}\label{ems} 0 \to H^0_B(S) \to S \to \bigoplus_{\alpha \in \Cl(Y)} 
H^0(X, \CO_X(\alpha)) \to H^1_B (S) \to 0 \, .\end{equation}
We therefore get the following result for subschemes of toric
varieties.

\begin{corollary} \label{TVcor} Let $X$ be a subscheme of a toric variety $Y$
  corresponding to a homogeneous ideal $I$ in the Cox ring $R$ of $Y$. Set
  $S = R/I$ and let $Z$ be the intersection of the singular locus of $Y$ with
  $X$ and $U= X \setminus Z$. Assume $I$ is generated by homogeneous
  polynomials of degrees $\alpha_1, \dots , \alpha_m \in Cl(Y)$. If
\begin{list}{\textup{(\roman{temp})}}{\usecounter{temp}}
\item $\depth_Z  \CO_X \ge 2$ and  $H^0(X, \CO_X (\alpha)) \simeq H^0(U, {\CO_X}_{\mid U}(\alpha))$ for every $\alpha \in Cl(Y)$ \rm(e.g. $\depth_Z \CO_X(\alpha) \ge 2$ for all $\alpha \in \Cl(Y)$ \rm) ,
\item $S_{\alpha_i} \simeq H^0(X, \CO_X(\alpha_i))$ for all $i = 1,
  \dots m$
\end{list} 
then $\Def^0_{S/R}$ and $\Hilb_{X/Y}$ are isomorphic deformation functors.
\end{corollary}
\begin{proof} The first statement is just a rewrite of Theorem \ref{hilbthm} (i) using $\mathcal{S}
= \bigoplus_{\alpha \in \Cl(Y)} \CO_X(\alpha)$. What is left is to show that $\Hom_R(I, H^1_{B}(S))_0 = 0$, where we mean degree $0$ in the $\Cl(Y)$ grading. Clearly $\varphi \in \Hom_R(I, H^1_{B}(S))$ is determined by its values on generators of $I$, so if $ H^1_{B}(S)_{\alpha_i} = 0$ for all $i = 1, \dots m$ we are done. But this follows from the statement in (ii) and the exact sequence \eqref{ems}.
\end{proof}

\begin{corollary} \label{wpscor} Let $X = \Proj S$ be a subscheme of a well
  formed weighted projective space $\PP(\mathbf{q}) = \Proj R$ defined by the
  homogeneous ideal $I$. Let $\mathfrak{m}$ be the irrelevant maximal ideal of
  $R$ and $Z$ the intersection of the singular locus of $\PP(\mathbf{q})$ with
  $X$. If
\begin{list}{\textup{(\roman{temp})}}{\usecounter{temp}}
\item $\depth_Z \CO_X \ge 2$ and
  $H^0(X, \CO_X (m)) \simeq H^0(U_X, {\CO_X}_{\mid U_X}(m))$ for every $m \in
  \ZZ$, \rm(e.g. $\depth_Z \CO_X(m) \ge 2$ for all $m \in \ZZ$ \rm) 
% $\depth_Z \bigoplus_{m \in \ZZ} \CO_X(m) \ge 2$
\item  $\Hom_R(I, H^1_{\mathfrak{m}}(S))_0 = 0$
\end{list} 
then $\Def^0_{S/R}$ and $\Hilb_{X/Y}$ are isomorphic deformation
functors.
\end{corollary}

We give some examples to illustrate
the conditions in Theorem \ref{hilbthm} and Corollary
\ref{TVcor}. Note that in the (multi)-graded case we may think of $\Def^0_{S/R}$ as deformations that preserve
the Hilbert function and that they correspond to 
Hilbert function strata of the Hilbert scheme, see \cite[Theorem
1.1]{kl:sm} and \cite{hs:mu}. 

\begin{example}\emph{Points.} If $X = \Proj S$ is $s$ points in general enough
  position in $\PP^n$, then the Hilbert function of $X$ is 
$$h_X(\nu) = \dim S_\nu = \inf \{s, \binom{\nu + n}{n}\}$$
by e.g\ \cite{gmr:hi}. Let $\nu_0$ be the smallest integer with 
$$s \le \binom{\nu_0 + n}{n}\, .$$
The exact sequence
$$0 \to S_\nu \to H^0(X, \CO_X(\nu)) \to H^1(\mathcal{I}_X(\nu)) \to
0 $$
and the fact that $h^0(\CO_X(\nu)) = s$ yield $I_\nu = 0$ for $\nu <
\nu_0$ and $H^1(\mathcal{I}_X(\nu)) \simeq H^1_\mathfrak{m}(S)_\nu = 0$ for
$\nu \ge \nu_0$ so by Corollary~\ref{wpscor}, 
$\Def^0_{S/R} \simeq \Hilb_{X/\PP^n}$ for points in general enough
  position.

  Six general points in $\PP^2$ will have Hilbert function $(1,3,6,6, \dots)$
  and hence $\Def^0_{S/R} \simeq \Hilb_{X/\PP^n}$. On the other hand the
  complete intersection of a quadric and a cubic will have
  $h_X = (1,3,5,6,6,\dots)$. Thus $\dim I_2 = h^1(\mathcal{I}_X(2)) = 1$, so
  $\Hom(I, H^1_{\mathfrak{m}}(S))_0 \simeq k$ and the functors need not be
  isomorphic, cf. Corollary~\ref{wpscor} (ii). Indeed since both obstruction
  spaces vanish ($S$ is a complete intersection) the map
  $\Def^0_{S/R} \to \Hilb_{X/\PP^n}$ corresponds dually to a surjection of
  formally smooth complete $k$-algebras. The above two cases show that the
  Hilbert function stratum given by $(1,3,6,6, \dots)$ is an open subscheme of
  $\Hilb^6(\PP^2)$ while $(1,3,5,6, \dots)$ gives a smooth codimension 1
  stratum.
\end{example}

\begin{example}\label{cex} \emph{Curves.} 
  Consider a smooth curve $C$ sitting on a hypersurface $V$ of degree
  $s < \sum_{i=0}^3 q_i$ in $\PP(q_0,q_1,q_2,q_3)$. Let
  $C = \Proj S \subset \PP({\bf q}) = \Proj R$ with $S = R/I$ and suppose that
  $\Spec R/I_V \setminus \{0\}$ is smooth, i.e. $V$ quasi-smooth. Applying
  $\mathcal{H}om(-,\CO_C)$ to the sequence 
  $0 \to \CI_{V/ \PP({\bf q})} \to \CI_{C/ \PP({\bf q})} \to \CI_{C/V} \to 0$ we get an
  exact sequence of normal bundles 
$$0 \to \CN_{C/V} \to \CN_{C/ \PP({\bf q})} \to \CO_C(s) \to 0$$
(cf. \cite[p. 93]{ha:def}) and
$\CN_{C/V} \simeq (\omega_C^{-1} \otimes \omega_V)^{-1}$. Thus by Serre
duality $H^1(\CN_{C/V} ) = 0$ and
$H^1(\CN_{C/ \PP({\bf q})}) \simeq H^1(\CO_C(s))$ while
$H^0(\CN_{C/ \PP({\bf q})}) \to H^0(\CO_C(s))$ is surjective.

Applying this to the long exact sequence in Theorem \ref{supp} we get
\begin{equation}  \label{grdefinhilb}
0 \to \Hom_R(I,S)_0 \to H^0(\CN_{C/ \PP({\bf q})}) \to \Hom_R(I,
H^1_\mathfrak{m}(S))_0 \to T^2(S/R;S)_0 \to 0 \, ,
\end{equation} because $
T^2_\mathfrak{m}(S/R;S)_0 \simeq \Hom_R(I, H^1_\mathfrak{m}(S))_0$ and the
composition 
$$H^1(\CN_{C/ \PP({\bf q})}) \to T^3_\mathfrak{m}(S/R;S)_0 \to \Hom_R(I,
H^2_\mathfrak{m}(S))_0 \to H^1(\CO_C(s))$$ is injective. Similarly there is a
commutative diagram
\begin{equation*}
\begin{tikzcd} H^0(\CN_{C/ \PP({\bf q})}) \arrow[twoheadrightarrow]{d}
\arrow{r} &  \Hom_R(I,H^1_\mathfrak{m}(S))
\arrow{d}\\H^0(\CO_C(s)) \arrow[twoheadrightarrow]{r}&
H^1(\CI_C(s))
\end{tikzcd}
\end{equation*}
so if $\Hom_R(I,H^1_\mathfrak{m}(S))_0 \simeq H^1(\CI_C(s))$ then
$T^2(S/R; S)_0=0$ and $\Def^0_{S/R}$ is unobstructed. If
$d_0, d_1, \dots ,d_m$ are the degrees of the minimal generators of $I$ and
$s=d_0$, then it is easy to see that there is such an isomorphism if
$H^1(\CI_C(d_i)) = 0$ for all $i = 1, \dots , m$.

If $ \PP({\bf q})= \PP^3$ and $s=2$ then $V \simeq \PP^1 \times \PP^1$ and we
can use the K\"{u}nneth formula to compute these cohomology groups. The
outcome is that the above vanishing holds and we get
$\Hom_R(I,H^1_\mathfrak{m}(S))_0 \simeq H^1(\CI_C(2))$, so $T^2(S/R;S)_0= 0$.
Moreover if $C$ has bidegree $(p,q)$ with $2 \le p \le q$ then
$$h^1(\CI_C(2)) = \begin{cases} 0 &\text{if } p > 2\\ 
\max \{0,q-3\}  & \text{if } p = 2\end{cases} \, .$$
Thus $\Def^0_{S/R} \simeq \Hilb_{C/\PP^3}$ are smooth if $p > 2$ or $q=3$ while
$\Def^0_{S/R}$ corresponds to a smooth stratum of codimension $q-3$ in
$\Hilb^{d,g}(\PP^3)$ otherwise, see \cite{ta:ca}.

If $s=3$ then among curves on a cubic surface in $\PP^3$ we find the curve
that gives rise to Mumford's example of an irreducible component $W$ of
$\Hilb^{14,24}(\PP^3)$ which is not reduced at its generic point
(\cite{mu:fu}). Let $H$ be a hyperplane section of $V$ and $E$ a line on $V$.
Take $C$ to be a generic element of the complete linear system $|4H+2E|$ on
$V$. Such a curve may be constructed as the linked curve to the curve
consisting of two disjoint conics in a general $(3,6)$ complete intersection,
and there is a resolution
$$0 \to R(-9) \to R(-8)^2 \oplus R(-7)^2 \to R(-6)^3 \oplus R(-3) \to
I \to 0 \, ,$$
cf.\ \cite{cu:mum}. One computes that $h^1(\CI_C(d))= 0$ when $d \ge 6$ and
$h^1(\CI_C(3)) = h^1(\CO_C(3)) = 1$, so
$\Hom_R(I,H^1_\mathfrak{m}(S)) \simeq H^1(\CI_C(3)) \simeq k$ and
$T^2(S/R;S)_0= 0$ in \eqref{grdefinhilb}. Thus $\Def^0_{S/R}$ is represented
by a formally smooth complete $k$-algebra which by \eqref{grdefinhilb}
corresponds to the reduced subscheme of the component $W$ at its generic
point.
\end{example}

\section{Deformations of the scheme - $\Def_X$} \label{Defsec}
In this section we must assume that the characteristic of the ground field
is $0$. Fix for the whole of this section  a $G$-quadruple $(X,Z,S,J)$ 
with associated sheaf of algebras $\CS$. Moreover fix $M$, a finitely generated  $SG$-module, and set $\CF =
\pi_\ast (\widetilde{M}_{\mid X^\prime})$ to be the corresponding
sheaf of  $\CS G$-module on $X$. This section is mostly
concerned with computing cotangent groups with values in $\CS$ or
$\CO_X$, but because of future applications  it is
better to be a little more general and consider values in $M$, $\CF$
and $\CF^G$.

\subsection{Conditions for $\Def^G_{S} \to \Def_{X}$ to be smooth
or an isomorphism}
\begin{lemma}\label{defGlemma} If $T^2(\CS/ \CO_X;\CS)^G = 0$
and $T^2_J(S/k;S)^G = 0$ then $\Def^G_{S} \to \Def_{X}$ is smooth. If moreover
$T^1(\CS/ \CO_X;\CS)^G= 0$ and $T^1_J(S/k;S)^G = 0$ then $\Def^G_{S}
\to \Def_{X}$  is an isomorphism.  
\end{lemma}
\begin{proof} First note that since taking $G$ invariants is exact we
have $T^i_X(\CS)^G \simeq T^i_X$. Secondly
\cite[Appendice, Proposition 56]{an:hom} and the fact that $\pi$ is
affine imply $T^i_X(\CS) \simeq T^i(\pi^{-1}
\CO_X/k;\CO_{X^\prime})$. The Zariski-Jacobi sequence for $k \to
\pi^{-1} \CO_X \to \CO_{X^\prime}$ reads
$$ \to T^i(\CO_{X^\prime}/\pi^{-1} \CO_X;\CO_{X^\prime})  \to
T^i_{X^\prime} \to T^i(\pi^{-1} \CO_X/k;\CO_{X^\prime}) \to
T^{i+1}(\CO_{X^\prime}/\pi^{-1} \CO_X;\CO_{X^\prime}) \to $$ so the
first condition yields $(T^i_{X^\prime})^G \to T^i_X$ surjective for
$i=1$ and injective for $i=2$.  The exact sequence in Theorem
\ref{supp} for $X^\prime \subseteq \Spec S$ is
$$ \dots \to T^i_J(S/k;S)  \to
T^i_S \to T^i_{X^\prime} \to T^{i+1}_J(S/k;S) \to \cdots$$ so the
second condition implies the same for $(T^i_S)^G \to
(T^i_{X^\prime})^G$ and we have proven the first statement. 

 To prove the second statement we need to show that the conditions
 imply $\Def^G_{S}(A)
\to \Def_{X}(A)$  injective for all objects $A$ of $\mathbf{C}$.  More
explicitly we must show that if $S^1_A$ and $S^2_A$ are invariant
deformations over $A$ mapping to $X_A^1$ and $X_A^2$, then for every
isomorphism of deformations $\varphi_{X,A}: X_A^1 \to X_A^2$ there
exists an equivariant isomorphism $\varphi_{S,A}: S_A^1 \to S_A^2$
inducing $\varphi_{X,A}$.

If
$T^1(\CS/ \CO_X;\CS)^G= 0$ and $T^1_J(S/k;S)^G = 0$ then the same
argument as above shows that  $(T^i_S)^G \to
T^i_X$  is surjective for
$i=0$ and injective for $i=1$. We will prove the statement by
induction on the length of $A$. Let $\rho : B \to A$ be a small
extension and assume the statement is true for $A$. Let $S^1_B$
and $S^2_B$ be invariant 
deformations of $S$ over $B$ mapping to $X_B^1$ and $X_B^2$. Assume
$\varphi_{X,B}: X_B^1 \to X_B^2$ is an isomorphism of
deformations. Let $\varphi_{X,A}: X_A^1 \to X_A^2$ be the induced
isomorphism where $X_A^i = X_B^i \times_{\Spec B} \Spec A$. By
induction there exists an isomorphism $\varphi_{S,A}: S_A^1 \to S_A^2$
inducing $\varphi_{X,A}$ and where $S_A^i = S_B^i
\otimes_{B} A$.  

The  obstruction to lifting $\varphi_{S,A}$ to an isomorphism
$S_B^1 \to S_B^2$ is in $(T^1_S)^G \otimes \ker \rho$
and maps to $0$ in $T^1_X \otimes \ker \rho$ because a lifting of
$\varphi_{X, A}$ exists (namely  $\varphi_{X, B}$). Thus by
injectivity at the $T^1$ level the obstruction vanishes and there
exists a lifting $\varphi_{S,B}^\prime:  S_B^1 \to S_B^2$. It maps to
an isomorphism 
$\varphi_{X,B}^\prime$ which may differ from our given
$\varphi_{X,B}$. However $\varphi_{X,B} - \varphi_{X,B}^\prime$
defines an element of $T^0_X \otimes \ker \rho$ and the surjectivity
at the $T^0$ level yields an element $D_{S,B} \in  (T^0_S)^G \otimes
\ker \rho$ mapping to $\varphi_{X,B} -
\varphi_{X,B}^\prime$. Thus $\varphi_{S,B}:= \varphi_{S,B}^\prime +
D_{S,B}$ is an isomorphism inducing $\varphi_{X,B}$.
\end{proof}

We will later need a statement with more
general values. The proof is the same as the first part above.
\begin{lemma}\label{Mlemma} If $T^2(\CS/ \CO_X;\CF)^G = 0$
and $T^2_J(S/k;M)^G = 0$ then there is a surjective morphism
$T^1_S(M)^G \twoheadrightarrow T^1_X(\CF^G)$.
\end{lemma}

The modules $T^i_J(S/k;M)$ are easily described in terms of local cohomology.
\begin{lemma}\label{lclemma} If $\depth_J M \ge 1$ then there is an isomorphism 
$T^1_J(S/k;M) \simeq \Der_k(S, H^1_J(M))$ and an exact
sequence
$$0 \to T^1(S/k; H^1_J(M)) \to T^2_J(S/k;M) \to T^0(S/k; H^2_J(M)) \to
T^2(S/k; H^1_J(M)) \, .$$ In particular if $\depth_J M \ge 2$ or $S$
is regular then
$T^2_J(S/k;M) \simeq \Der_k(S, H^2_J(M))$ as $SG$-modules.
\end{lemma}
\begin{proof} Since $H^0_J(M) = 0$ the exact sequence comes from the
edge exact sequence for the spectral sequence in Theorem
\ref{supp}. \end{proof}

\subsection{Euler derivations}

We show that if $\depth_Z\CS \ge 2$ then the relative tangent sheaf
$\Theta_{X'/X}$ is globally free of rank equal to the dimension of $G$. This
will allow us to give finer criteria for the equivalence of deformation
functors. The result follows from general considerations and we start with
some lemmas.

Let $G$ be a reductive group acting on an affine scheme $X = \Spec A$
over a field $k$ of characteristic 0. Let $\mathfrak{g} = T_e$ be the Lie
algebra of $G$ (where $e$ is the identity element in $G$). Consider
the corresponding representation of 
Lie algebras $\phi: A \otimes \mathfrak{g} \to
\Der_{A^G}(A,A)$. Recall that if $\xi \in
\mathfrak{g} = \Der_k(k[G], k(e))$ then $\phi(1 \otimes \xi)$   is the
composition
\begin{equation} A \xrightarrow{\mu} A \otimes k[G]  \xrightarrow{1
    \otimes \xi} A \otimes k(e) \simeq A
\end{equation}
where $\mu$ is the comultiplication of the group action. Note that if
$f$ is invariant then $\mu(f) = f \otimes 1$ so the composition is in
$\Der_{A^G}(A,A)$. 

If $D \in \Der_{A^G}(A,A)$ and $x$ is a closed point of $\Spec A$, let
$D_x$ be the value in 
$\Der_{A^G}(A,k(x))$.  For a fixed closed point $x$ let $p_x : G \to
Gx$ be the orbit map. Then $\phi(1 \otimes \xi)_x(f) = \xi(p_x^{\#}(f))$ so $\phi(\xi)_x$
equals the image of $\xi$ via the map of
Zariski tangent spaces $d p_x :T_e \to T_{Gx,x}$. Note that $T_{Gx,x}
= \Der_k(\CO_{Gx}, k(x)) \subseteq \Der_{A^G}(A, k(x)) \subseteq  \Der_{k}(A, k(x))= T_{X,x}$.

\begin{lemma}\label{inj}  If there is an open dense
  subset of $\Spec A$ where all isotropy groups are finite, then
  $\phi: A \otimes_k \mathfrak{g} \to 
\Der_{A^G}(A,A)$ is injective.
\end{lemma}
\begin{proof} Choose a basis $\xi_1, \dots ,\xi_r$ for $\mathfrak{g}$
  and assume $\phi(\sum a_i \otimes \xi_i) = 0$. Then $\sum a_i(x)
  dp_x(\xi_i) = 0$ in $T_{Gx,x}$ for all $x$. But by assumption $dp_x$
  is an isomorphism on an open dense subset $U$. Thus for all $i = 1,
  \dots , r$ we have $a_i = 0$ on $U$ so $a_i = 0$ in $A$.
\end{proof}

\begin{lemma} \label{isolemma} If the quotient map $\Spec A \to \Spec A^G$ is smooth of
  relative dimension $\dim G$ and a geometric quotient then
  $\phi: A \otimes_k \mathfrak{g} \to 
\Der_{A^G}(A,A)$ is an isomorphism.
\end{lemma}
\begin{proof}  Injectivity follows from Lemma \ref{inj}. We will show
  surjectivity for the maps of stalks $A_x \otimes \mathfrak{g} \to
  \Der_{A^G}(A,A)_x = \Der_{A_y^G}(A_x,A_x)$ where $x \mapsto y$ via
  the quotient map. By assumption $\CO_{Gx}
  \simeq A \otimes_{A^G} k(y)$ and so $\Der_{A_y^G}(A_x,k(x)) \simeq
  \Der_{k}(\CO_{Gx},k(x)) = T_{Gx, x} \simeq \mathfrak{g}$. On the
  other hand since $A^G \to A$ is smooth we have
  $\Der_{A_y^G}(A_x,k(x)) \simeq \Der_{A_y^G}(A_x,A_x) \otimes_{A_x}
  k(x)$ and so the surjectivity follows from Nakayama's Lemma.
\end{proof}

We write $M^\vee = \Hom_A(M,A)$ for the dual module. The above
construction yields also a map $\psi : \Omega_{A/A^G} \to
\Hom_k(\mathfrak{g}, A)$ namely the composition
$$\Omega_{A/A^G} \to \Omega_{A/A^G}^{\vee \vee}=
\Hom_A(\Der_{A^G}(A,A), A) \xrightarrow{\phi^\vee}
\Hom_k(\mathfrak{g}, A) \, .$$
Thus $\psi(da)(\xi) = \phi(1 \otimes \xi)(a)$

\begin{lemma} \label{omsurj} If $x$ is a closed point in $\Spec A$ mapping to $y \in
  \Spec A^G$ with finite
  isotropy group and $Gx$ is closed in $\Spec A$ then $\psi_x : \Omega_{A_x/A_y^G} \to
\Hom_k(\mathfrak{g}, A_x)$ is surjective.
\end{lemma}
\begin{proof} By Nakayama's Lemma it is enough to prove
  $\Omega_{A_x/A_y^G} \otimes_{A_x} k(x) \to
\Hom_k(\mathfrak{g}, k(x))$ is surjective, which we can do by showing
that the $k$-dual map is injective. But this is the map $\mathfrak{g}
\otimes k(x) \to  \Der_{A_y^G}(A_x,k(x))$ induced by $\phi$. Thus it
factors in the following way $$\mathfrak{g} \otimes k(x) \simeq T_{Gx,x} \subseteq
\Der_{A_y^G}(A_x,k(x))\, .$$
\end{proof}

We now apply these local lemmas to our situation. By $\CF \otimes_k
\mathfrak{g}$ we mean the $\CS G$-module that as $\CS$-module is
isomorphic to the sum of $\dim \mathfrak{g}$ copies of $\CF$ and the
$G$-action is given by $g(f \otimes \xi) = gf \otimes \Ad_g(\xi)$
where $\Ad_g:\mathfrak{g} \to \mathfrak{g}$ is the adjoint
action. Similarly we can define $\CF \otimes_k
\mathfrak{g}^\ast$ for the dual of the adjoint action.
\begin{lemma}\label{eu} If 
$\depth_Z\CS \ge 2$ then there is an $\CS G$ isomorphism 
$$\mathcal{T}^0(\CS/\CO_X; \CS) \simeq
\mathcal{S} \otimes_k \mathfrak{g}\, .$$ 
More generally if $\depth_Z \CF \ge 2$ then $\mathcal{T}^0(\CS/\CO_X; \CF) \simeq
\mathcal{F} \otimes_k \mathfrak{g}$ as $\CS G$-modules.
\end{lemma}
\begin{proof} By definition  $\pi_{\mid U^\prime}$ satisfies locally
 the conditions in Lemma \ref{isolemma} . This globalizes to an
 isomorphism $\Theta_{U'/U} \simeq \CO_{U'} \otimes_k
 \mathfrak{g}$. Since $\pi$ is smooth on $U^\prime$ we get
 $\Omega^1_{U^\prime/U} \simeq \CO_{U'}  \otimes \mathfrak{g}^\ast$. Since
 $\mathcal{T}^0(\CS/\CO_X; \CF) =
 \mathcal{H}om(\Omega^1_{\CS/\CO_X}, \CF)$ it also has depth $\ge 2$ in $Z$,
 cf. Theorem~\ref{supp}. It is
therefore isomorphic to $j_\ast 
\mathcal{H}om(\pi_\ast\Omega^1_{U^\prime/U} , \CF_{\mid U}) \simeq \CF
\otimes_k \mathfrak{g}$ where $j$ is the inclusion $U \subseteq X$.
\end{proof}

For example let $G = \mathbb{G}_m$ and $A$ be a finitely generated
algebra, or a localization of such, with algebra generators $x_1,
\dots , x_n$. Then $\mu: A \to A \otimes k[t, t^{-1}]$ is of the form
$f(x_1, \dots , x_n) \mapsto f(t^{a_1}x_1, \dots , t^{a_n}x_n)$ for
integers $a_1, \dots , a_n$. We get $$\phi(1
\otimes \frac{d}{dt}_{\mid t=1})(f) =  \frac{d}{dt} f(t^{a_1}x_1,
\dots , t^{a_n}x_n)_{\mid t=1} = \sum_{i=1}^n a_i x_i \frac{\partial
  f}{\partial x_i}$$ 
by the chain rule. Thus the generator of $\mathfrak{g}$ maps to the
Euler derivation of the $\ZZ$-graded algebra $A$. This motivates the
name in the following definition.

\begin{definition} \label{eudef} Assuming $\depth_Z\CS \ge 2$, let
$E_1, \dots, E_r$ be a set of global sections that generate
the free sheaf $\mathcal{T}^0(\CS/\CO_X; \CS)$. We call such
a set a \emph{set of Euler derivations} for the $G$-quadruple. A
set of Euler derivations defines an $\CS G$-map $E:
\Omega^1_{\CS/\CO_X} \to \CS \otimes \mathfrak{g}^\ast$. Let
$\mathcal{Q} = \coker(E)$ and $Z(\mathcal{Q})$ be the support of $\mathcal{Q}$. 
\end{definition}

The morphism $E$ is just the natural map
$\Omega^1_{\CS/\CO_X} \to (\Omega^1_{\CS/\CO_X})^{\vee\vee}$ after
choosing a basis for the free sheaf
$(\Omega^1_{\CS/\CO_X})^{\vee\vee}$ thus we have an exact sequence 
\begin{equation} \label{euseq} 0 \to
\mathcal{T}or(\Omega^1_{\CS/\CO_X}) \to \Omega^1_{\CS/\CO_X}
\xrightarrow{E} \CS \otimes \mathfrak{g}^\ast
\to \CQ \to 0 \end{equation} 
where $\mathcal{T}or(\Omega^1_{\CS/\CO_X})$
is the torsion
submodule of
$\Omega^1_{\CS/\CO_X}$.

If $\depth_Z\CS \ge 2$ and the sheaf $\mathcal{T}^1(\CS/\CO_X; \CS)^G
= 0$ then the Zariski-Jacobi sequence for $k \to \CO_X \to \CS$ leads
to a short
exact sequence  
$$0 \to \mathcal{T}^0(\CS/\CO_X; \CS)^G \to (\mathcal{T}^0_\CS)^G \to
\mathcal{T}^0(\CO_X/k; \CS)^G \to 0\, .$$
Therefore  by Lemma \ref{eu} we can make the following definition.
\begin{definition} \label{defeuseq}If $\depth_Z\CS \ge 2$ and the sheaf
  $\mathcal{T}^1(\CS/\CO_X; \CS)^G = 0$ then we call the short 
exact sequence 
$$0 \to (\CS \otimes \mathfrak{g})^G \to \Theta_\CS^G \to
\Theta_X \to 0$$ the \emph{Euler sequence} associated to the
$G$-quadruple. Here $\Theta_X = \mathcal{T}^0(\CO_X/k; \CO_X)$.
\end{definition}

\begin{remark} If $G \to \GL(V)$ is a representation and $S = \Sym
  (V^\ast)$ then $\Der_k(S) \simeq S \otimes V$ as $SG$-module. From
  the above we have an $SG$-map  $S \otimes \mathfrak{g} \to S \otimes
  V$. In this case it is induced from the natural map $\mathfrak{g} \to
  \Hom(V,V) \simeq \Sym^1(V^\ast) \otimes V$ given by the representation.
\end{remark}

\begin{lemma}\label{tiz} If
$\depth_Z\CF \ge 2$ then $$\mathcal{T}^i(\CS/\CO_X; \CF) \simeq
\mathcal{E}xt^i_\CS(\Omega^1_{\CS/\CO_X},\CF) \text{ for $i =
0,1,2$. }$$
\end{lemma}
\begin{proof} Consider the spectral sequence
$\mathcal{E}xt^p_\CS(\mathcal{T}_q(\CS/\CO_X; \CS),\CF) \Rightarrow
\mathcal{T}^{p+q}(\CS/\CO_X; \CF)$. Since the
$\mathcal{T}_q(\CS/\CO_X; \CS)$ are supported on $Z$ when $q \ge 1$,
the depth condition yields the result.
\end{proof}

Assume $\CS$ is a sheaf of $\CO_X$-algebras on a
scheme $X$ and $Z \subseteq X$ is  locally closed. If $\mathcal{F}$
and $\mathcal{G}$ are $\CS$-modules, denote by $\Ext^i_{\CS,
Z}(\mathcal{F}, \mathcal{G})$, respectively $\mathcal{E}xt^i_{\CS,
Z}(\mathcal{F}, \mathcal{G})$, the higher derived functors of
$\mathcal{G} \mapsto H^0_Z (X, \mathcal{H}om_\CS(\mathcal{F},
\mathcal{G}))$, respectively $\mathcal{G} \mapsto \mathcal{H}^0_Z(X,
\mathcal{H}om_\CS(\mathcal{F}, \mathcal{G}))$. We refer to SGA 2
Expos\`{e} VI (\cite{gr:sga2}) for details and the results we will
use.

\begin{lemma}\label{t1rel} Assume $\depth_Z\CF \ge 2$. 
\begin{list}{\textup{(\roman{temp})}}{\usecounter{temp}}
\item There is an isomorphism $\mathcal{E}xt^1_\CS(\Omega^1_{\CS/\CO_X},\CF) \simeq
\mathcal{H}om_\CS(\CQ, \mathcal{H}^2_Z(\CF))$.
\item  If $\depth_Z (\CF
\otimes \mathfrak{g})^G \ge
3$ or $\CQ = 0$ then $$\mathcal{T}^1(\CS/\CO_X; \CF)^G =
\mathcal{E}xt^1_\CS(\Omega^1_{\CS/\CO_X},\CF)^G
=\mathcal{H}om_\CS(\CQ, \mathcal{H}^2_Z(\CF))^G =0 \, .$$ 
\end{list} 
\end{lemma}
\begin{proof} By Lemma \ref{eu},
  $\mathcal{H}om(\Omega^1_{\CS/\CO_X} , \CF)
\simeq \CF \otimes \mathfrak{g}$ so we get an exact
sequence
$$\cdots  \to \mathcal{E}xt^1_{\CS,
Z}(\Omega^1_{\CS/\CO_X}, \CF) \to
\mathcal{E}xt^1_{\CS}(\Omega^1_{\CS/\CO_X}, \CF) \to
\mathcal{R}^1j_\ast(\CF \otimes \mathfrak{g})_{\mid U}
\xrightarrow{\varepsilon} \mathcal{E}xt^2_{\CS,
Z}(\Omega^1_{\CS/\CO_X}, \CF) \to \cdots $$ from the long exact
sequence for $\mathcal{E}xt^i_{\CS, Z}(\Omega^1_{\CS/\CO_X},
\CF)$. Here $j$ is the inclusion of $U$ in $X$. Now since $\depth_Z\CF
\ge 2$ the left sheaf vanishes. The sheaf
$\mathcal{R}^1j_\ast(\CF \otimes \mathfrak{g})_{\mid U}$ is isomorphic to
$\mathcal{H}^2_Z(\CF \otimes \mathfrak{g}) \simeq \mathcal{H}^2_Z(\CF)
\otimes \mathfrak{g}$. Using the spectral sequence
$\mathcal{E}xt^p_{\CS}(\Omega^1_{\CS/\CO_X}, \mathcal{H}^q_Z(\CF))
\Rightarrow \mathcal{E}xt^{p+q}_{\CS, Z}(\Omega^1_{\CS/\CO_X}, \CF)$
we get $$\mathcal{E}xt^2_{\CS, Z}(\Omega^1_{\CS/\CO_X}, \CF) \simeq
\mathcal{H}om_{\CS}(\Omega^1_{\CS/\CO_X}, \mathcal{H}^2_Z(\CF)) \, .$$
Thus our exact sequence becomes
$$0 \to
\mathcal{E}xt^1_{\CS}(\Omega^1_{\CS/\CO_X}, \CF) \to
\mathcal{H}^2_Z(\CF) \otimes \mathfrak{g} \xrightarrow{\varepsilon}
\mathcal{H}om_{\CS}(\Omega^1_{\CS/\CO_X}, \mathcal{H}^2_Z(\CF)) \to
\cdots $$ where $\varepsilon$ is the map induced by $E$ in the
sequence \eqref{euseq}. It follows
that $$\mathcal{E}xt^1_{\CS}(\Omega^1_{\CS/\CO_X}, \CF) = \ker
\varepsilon \simeq \mathcal{H}om_\CS(\CQ, \mathcal{H}^2_Z(\CF))$$
proving the first statement.

Clearly if $\CQ = 0$ then $\mathcal{E}xt^1_{\CS}(\Omega^1_{\CS/\CO_X},
\CF)= 0$ and the
inclusion $$\mathcal{E}xt^1_{\CS}(\Omega^1_{\CS/\CO_X}, \CF)^G
\hookrightarrow \mathcal{H}^2_Z(\CF \otimes \mathfrak{g})^G \simeq
\mathcal{H}^2_Z((\CF \otimes \mathfrak{g})^G)$$ yields the second
statement. \end{proof}

We now have conditions for the existence of an Euler sequence.
\begin{theorem} \label{thmeuseq} If $\depth_Z\CS \ge 2$ and
  $\mathcal{H}om_\CS(\CQ, \mathcal{H}^2_Z(\CS))^G = 0$  then there is an Euler sequence
$$0 \to (\CS \otimes \mathfrak{g})^G \to \Theta_\CS^G \to
\Theta_X \to 0 \, .$$
\end{theorem}

\begin{proposition} \label{zqpos} The closed subset $Z(\mathcal Q)
  \subset X$ is contained in 
  the set of points $x \in X$ where $\pi^{-1}(x)$ contains a point with
  positive dimensional isotropy group.
\end{proposition}
\begin{proof} We may prove the result locally on an affine chart,
  i.e.\ for $\pi: \Spec A \to \Spec A^G$. A
  standard result in invariant theory says that the stable points are
  the complement of $\pi^{-1}(\pi(L))$ where $L$ is the locus of
  points with positive dimensional isotropy. The result follows now
  from Lemma \ref{omsurj}.
\end{proof}

\begin{example} For toric varieties Proposition \ref{zqpos} says that $Z(\CQ)$
  is contained in the non-simplicial locus. In fact one can prove
  directly using Euler derivations that these loci are equal. In
  particular for simplicial toric varieties $\CQ = 0$ and we recover the Euler
  sequence of \cite[Theorem 12.1]{bc:ho}
\end{example}

\begin{example} For Grassmannians (Example \ref {grass}) even $Z$ is
  empty, Proposition \ref{zqpos} applies. If $\mathcal{L}$ is the
  universal subbundle and $\mathcal{P}$ is the universal quotient
  bundle then the tangent sheaf is $\mathcal{P} \otimes
  \mathcal{L}^\vee$. Our Euler sequence is just the usual sequence
  induced by the tautological exact sequence.
\end{example}

\subsection{The groups $T^i(\CS/ \CO_X;\CF)$} As we shall see in the computations here and the results in the next section the behavior of the sheaf $\CQ$, the cokernel of the Euler derivations (Definition \ref{eudef}) plays an important role when comparing deformation functors.

\begin{lemma}\label{tizgl} If $\depth_Z\CF \ge 2$ and
$\mathcal{H}om_\CS(\CQ, \mathcal{H}^2_Z(\CF))^G = 0$ then
$$T^i(\CS/\CO_X; \CF)^G \simeq
\Ext^i_\CS(\Omega^1_{\CS/\CO_X},\CF)^G \text{ for } i = 0,1,2 \, .$$
\end{lemma}
\begin{proof} Lemma \ref{t1rel} says that under the conditions above, 
$\mathcal{T}^1(\CS/\CO_X; \CF)^G =
\mathcal{E}xt^1_\CS(\Omega^1_{\CS/\CO_X},\CF)^G=0$.  We have local
global spectral sequences which at the $E_2$ level are
\begin{align*}H^p (X, \mathcal{T}^q(\CS/\CO_X; \CF)) &\Rightarrow
T^{p+q}(\CS/\CO_X; \CF) \\ H^p (X,
\mathcal{E}xt^q_\CS(\Omega^1_{\CS/\CO_X},\CF) ) &\Rightarrow
\Ext^{p+q}_\CS(\Omega^1_{\CS/\CO_X},\CF) \, .
\end{align*} Write (just for this proof) $ \mathcal{T}^i
=\mathcal{T}^i(\CS/\CO_X; \CF)$ and $\mathcal{E}^i
=\mathcal{E}xt^i_\CS(\Omega^1_{\CS/\CO_X},\CF)$.  The edge
homomorphisms for the spectral sequence
$\mathcal{E}xt^p_\CS(\mathcal{T}_q(\CS/\CO_X; \CS),\CF) \Rightarrow
\mathcal{T}^{p+q}(\CS/\CO_X; \CF)$ yield natural maps $\mathcal{E}^i
\to \mathcal{T}^i$.  Lemma \ref{tiz} implies $T^i(\CS/\CO_X; \CF)^G
\simeq \Ext^i_\CS(\Omega^1_{\CS/\CO_X},\CF)^G$ for $i=0,1$. Since
$\mathcal{E}xt^1_\CS(\Omega^1_{\CS/\CO_X},\CF)^G = 0$ the spectral
sequences yield a commutative diagram with exact rows
$$
\begin{tikzcd} 0 \arrow{r} & H^2 (X, \mathcal{E}^0 )^G \arrow{r}
\arrow{d} & \Ext^{2}_\CS(\Omega^1_{\CS/\CO_X},\CF)^G \arrow{r}
\arrow{d} & H^0(X, \mathcal{E}^2)^G \arrow{r}\arrow{d} & H^3(X,
\mathcal{E}^0 )^G \arrow{d} \\ 0 \arrow{r} & H^2(X, \mathcal{T}^0)^G
\arrow{r} & T^{2}(\CS/\CO_X; \CF)^G \arrow{r} & H^0(X,
\mathcal{T}^2)^G \arrow{r} & H^3(X, \mathcal{T}^0)^G
\end{tikzcd}
$$
so the isomorphism for $i=2$ follows from Lemma \ref{tiz} and the Five
Lemma.
\end{proof}

\begin{proposition}\label{h30} Assume $\depth_Z\CF \ge 2$
\begin{list}{\textup{(\roman{temp})}}{\usecounter{temp}}
\item If $\mathcal{H}om_\CS(\CQ, \mathcal{H}^2_Z(\CF))^G = 0$ then
$T^1(\CS/\CO_X; \CF)^G \simeq H^1(X, (\CF \otimes \mathfrak{g})^G)$
and there is an exact sequence
\begin{multline*} \ 0 \to H^2(X, (\CF \otimes \mathfrak{g})^G) \to
T^2(\CS/\CO_X; \CF)^G \to H^0(X,
\mathcal{E}xt^2_\CS(\Omega^1_{\CS/\CO_X},\CF)^G) \\ \to H^3(X, (\CF
\otimes \mathfrak{g})^G) \end{multline*}
\item If $\depth_Z (\CF \otimes \mathfrak{g})^G \ge 3$ then there is
an exact sequence
$$ 0 \to \Hom_\CS(\Omega^1_{\CS/\CO_X}, \mathcal{H}^2_Z(\CF))^G
\to H^0(X, \mathcal{E}xt^2_\CS(\Omega^1_{\CS/\CO_X},\CF)^G) \to
\Hom_\CS(\CQ, \mathcal{H}^3_Z(\CF))^G \, .$$
\item If $\CQ = 0$ then $
\mathcal{E}xt^2_\CS(\Omega^1_{\CS/\CO_X},\CF) \simeq
\mathcal{H}om_\CS(\mathcal{T}or(\Omega^1_{\CS/\CO_X}),
\mathcal{H}^2_Z(\CF))$.
\end{list}
\end{proposition}
\begin{proof} Since the condition in (i) implies
$\mathcal{T}^1(\CS/\CO_X; \CF)^G = 0$ by Lemma \ref{tiz} and Lemma
\ref{t1rel} (i) follows from the 
local-global spectral sequence and Lemma \ref{eu} and Lemma
\ref{tizgl}.

Let $j$ be the inclusion of $U$ in $X$.  To prove (ii) we consider the
long exact sequence for $\mathcal{E}xt^i_{\CS,
Z}(\Omega^1_{\CS/\CO_X}, \CF)$ as in the proof of Lemma \ref{t1rel},
and get
\begin{multline} \label{exzseq} \mathcal{H}^2_Z(\CF \otimes
\mathfrak{g}) \to \mathcal{E}xt^2_{\CS, Z}(\Omega^1_{\CS/\CO_X}, \CF)
\to \mathcal{E}xt^2_{\CS}(\Omega^1_{\CS/\CO_X}, \CF) \to
\mathcal{H}^3_Z(\CF \otimes \mathfrak{g}) \\ \xrightarrow{u}
\mathcal{E}xt^3_{\CS, Z}(\Omega^1_{\CS/\CO_X}, \CF) \, .
\end{multline}Taking invariants and using the depth assumption this
yields
$$ 0 \to \mathcal{E}xt^2_{\CS,
Z}(\Omega^1_{\CS/\CO_X}, \CF)^G \to
\mathcal{E}xt^2_{\CS}(\Omega^1_{\CS/\CO_X}, \CF)^G \to
\mathcal{H}^3_Z(\CF \otimes \mathfrak{g})^G \xrightarrow{u}
\mathcal{E}xt^3_{\CS, Z}(\Omega^1_{\CS/\CO_X}, \CF)^G \, .$$ Using the
spectral sequence $\mathcal{E}xt^p_{\CS}(\Omega^1_{\CS/\CO_X},
\mathcal{H}^q_Z(\CF)) \Rightarrow \mathcal{E}xt^{p+q}_{\CS,
Z}(\Omega^1_{\CS/\CO_X}, \CF)$ we get $$\mathcal{E}xt^2_{\CS,
Z}(\Omega^1_{\CS/\CO_X}, \CF)^G \simeq
\mathcal{E}xt^0_{\CS}(\Omega^1_{\CS/\CO_X}, \mathcal{H}^2_Z(\CF))^G \,
.$$ Moreover there is a composite map $\varepsilon = v \circ u$
$$ \mathcal{H}^3_Z(\CF \otimes \mathfrak{g}) \xrightarrow{u} \mathcal{E}xt^3_{\CS,
Z}(\Omega^1_{\CS/\CO_X}, \CF) \xrightarrow{v}
\mathcal{H}om_\CS(\Omega^1_{\CS/\CO_X}, \mathcal{H}^3_Z(\CF))$$ where
$v$ is the map to $E_2^{0,3}$ in the above spectral sequence.  Thus
$\ker u \subseteq \ker \varepsilon$.  On the other hand $\varepsilon$
is the map induced by $E$ in the exact sequence
\eqref{euseq}. Therefore $\ker \varepsilon = \mathcal{H}om_\CS(\CQ,
\mathcal{H}^3_Z(\CF))$ and we get the exact sequence in the statement.

To prove (iii) note that if $\CQ = 0$ then even without the depth
assumption the above argument shows that $\ker \varepsilon = 0$. Thus
$u$ is injective. Moreover by Lemma \ref{t1rel},
$\mathcal{E}xt^1_{\CS}(\Omega^1_{\CS/\CO_X}, \CF)$ vanishes so
\eqref{exzseq} becomes
$$ 0 \to  \mathcal{H}^2_Z(\CF \otimes
  \mathfrak{g}) \to \mathcal{H}om_\CS(\Omega^1_{\CS/\CO_X},
\mathcal{H}^2_Z(\CF)) \to
\mathcal{E}xt^2_\CS(\Omega^1_{\CS/\CO_X},\CF) \to 0 \, .$$ The
sequence \eqref{euseq} is now short exact and we can apply
$\mathcal{H}om_\CS(-, \mathcal{H}^2_Z(\CF))$ to get an exact sequence
$$ 0 \to  \mathcal{H}^2_Z(\CF \otimes
  \mathfrak{g}) \to \mathcal{H}om_\CS(\Omega^1_{\CS/\CO_X},
\mathcal{H}^2_Z(\CF)) \to
\mathcal{H}om_\CS(\mathcal{T}or(\Omega^1_{\CS/\CO_X}),
\mathcal{H}^2_Z(\CF))\to 0 \, .$$ This proves (iii).
\end{proof}

\begin{remark} If $\CS$ is regular then $H^0(X,
\mathcal{E}xt^2_\CS(\Omega^1_{\CS/\CO_X},\CF)^G) \to \Hom_\CS(\CQ,
\mathcal{H}^3_Z(\CF))^G$ is surjective. See the proof of Theorem
\ref{afftor} below.
\end{remark}

\subsection{Results}
\begin{example} \label{fG} \emph{Finite $G$.} We can apply the above to the
situation in Example \ref{qs}, i.e. $G \subset \GL_n(\CC)$ is finite
without pseudo-reflections and 
$$(X,Z, S, J) = (\CC^n/G,\Sing(\CC^n/G),\CC[x_1, \dots , x_n],(1)) \, .$$ Assume that
$\codim_Z X \ge 3$ (the singularity need not be isolated). Then since
$X$ is Cohen-Macaulay, $\depth_Z \CO_X \ge 3$ and since $G$ is finite, 
$\depth_{I_Z} S \ge 3$ and $\CQ=0$.  It follows from Proposition
\ref{h30}, Lemma \ref{tiz} and Lemma \ref{defGlemma} that $X$ is
rigid. This was first proven by Schlessinger,  see \cite{sc:rig} and \cite{sc:onr}.

In general consider an affine $G$-quadruple
$$(X,Z, S, J) = (\Spec S^G,Z,S,(1))$$
with finite $G$ and $Z$ the locus where $\pi : \Spec S \to X$ is not
\'{e}tale. Then by the above results we have $\Def_S^G \simeq \Def_X$
if $\depth_{I_Z} S \ge 3$. In particular if $S$ is Cohen-Macaulay and
equidimensional then it
is enough to assume $\codim Z \ge 3$. See also \cite[Section
7]{ste:can}.
\end{example}

Our main result in its most general form is

\begin{theorem}\label{defGtheorem} Let $(X,Z,S,J)$ be a $G$-quadruple
  with associated sheaf of algebras $\CS$ and assume $\depth_Z\CS \ge 2$. If
\begin{itemize}
\item[\textup{(i)}] $\mathcal{H}om_\CS(\CQ, \mathcal{H}^2_Z(\CS))^G = 0$
\item[\textup{(ii)}] $H^0(X,
  \mathcal{E}xt^2_\CS(\Omega^1_{\CS/\CO_X},\CS)^G) = 0$
\item[\textup{(iii)}] $H^2(X, (\CS \otimes \mathfrak{g})^G) = 0$
\item[\textup{(iv)}] $T^1(S/k; H^1_J(S))^G = 0$ and $\Der_k(S, H^2_J(S))^G = 0$
\end{itemize} then $\Def^G_{S} \to \Def_{X}$ is smooth. If moreover
\begin{itemize}
\item[\textup{(v)}] $H^1(X, (\CS \otimes \mathfrak{g})^G) = 0$
\item[\textup{(vi)}] $\Der_k(S, H^1_J(S))^G = 0$
\end{itemize} then $\Def^G_{S}
\to \Def_{X}$  is an isomorphism.  
\end{theorem}
\begin{proof} The result follows directly from Lemma \ref{defGlemma}
and Proposition \ref{h30}.
\end{proof}

Using the same proposition and lemma we get a slight improvement if $\CQ = 0$.
\begin{theorem}\label{Q0thm} Let $(X,Z,S,J)$ be a $G$-quadruple
  with associated sheaf of algebras $\CS$ and assume $\depth_Z\CS \ge 2$. If
\begin{itemize}
\item[\textup{(i)}] $\CQ = 0$
\item[\textup{(ii)}] $\Hom_\CS(\mathcal{T}or(\Omega^1_{\CS/\CO_X}), \mathcal{H}^2_Z(\CS))^G = 0$
\end{itemize} 
and \textup{(iii)} and \textup{(iv)} of Theorem \ref{defGtheorem}
hold, then $\Def^G_{S} \to \Def_{X}$ is smooth. 
If moreover
\textup{(v)} and \textup{(vi)} of Theorem \ref{defGtheorem} hold
then  $\Def^G_{S}
\to \Def_{X}$  is an isomorphism.  
\end{theorem}

We give two examples to illustrate how the statement fails if the
assumptions are not met.

\begin{example} For smooth curves in $\PP^3$ we only need to verify
  (iv) of Theorem \ref{Q0thm} to conclude that
$\Def_S^0 \to \Def_X$ is smooth since $Z$ is empty and $(\CS \otimes
\mathfrak{g})^G = \CO_X$. We claim that a general space curve
$X = \Proj S$
in $\Hilb^{d,g}(\PP^3)$ satisfying $g < d+3$  also satisfy
(iv). Indeed in this case $J$ is the irrelevant maximal ideal $(x_0,..,x_3)$, thus
$H^2_J(S) \simeq \bigoplus_\nu H^1(X,O_X(\nu))$. We get $\Der_k(S,
H^2_J(S))_0 = 0$ provided $H^1(X,\CO_X(1))=0$. It 
is well known that a general curve with $g < d+3$ satisfies this property.

We need to show $T^1(S/k, H^1_J(S))_0 = 0$. Let $I$ be the
homogeneous ideal of $X$ in $\PP^3 = \Proj R$. Since $\Hom_R(I,H^1_J(S))_0
= T^1(S/R, H^1_J(S)) \to T^1(S/k, H^1_J(S))_0$ is surjective, it
suffices to show $\Hom(I,H^1_J(S))_0 = 0$. Space curves of maximal rank
(i.e. $H^1(\CI_X(v)) = 0$ provided $H^0(\CI_X(v)) \ne 0)$ necessarily
satisfy this property. The main theorem of Ellia and Ballico in \cite{be:max}
implies that the general curve in the range $g < d-3$ has maximal
rank. Thus the first four assumptions of the theorem are
satisfied. 

It follows that $\Def_S^0 \to
\Def_X$ is smooth in this range. But if $g > 0$ then $H^1(X, \CO_X)
\ne 0$, i.e.\ (v) does not hold  and
we do not expect  $\Def_S^0 \to \Def_X$ to be an isomorphism.
\end{example}

\begin{example} Let $X = \Spec A$ be the affine cone over
  $\mathbb{G}(2,4)$ in the Pl{\"u}cker embedding. It
  is a node in $\mathbb{A}^6$ so $\dim T^1_X = 1$. On the other hand
  $A = S^G$ where $S = k[x_{1,1}, \dots , x_{2,4}]$ and $G = \SL_2$
  corresponding to an affine $G$-quadruple $(X, \{0\}, S, (1))$. Of
  course $T^1_S = 0$. If $\mathfrak{m} \subset A$ is the ideal of
  $\{0\}$ in $X$, then the ideal
  $I_Z = \mathfrak{m}S \subset S$ is generated by the $2 \times 2$
  minors of a general 
  $2 \times 4$ matrix so $\depth_Z S = 3$. Thus conditions (i), (iii)
  and (iv) of Theorem \ref{defGtheorem} are satisfied. 

We compute $\Ext^2_S(\Omega^1_{S/A}, S)$. Let $f_{ij}$ be the minor
with columns $i$ and $j$. If $1 \le i < j 
\le 4$ write $\{k,l\} = \{1,2,3,4\} \setminus \{i,j\}$. 
Then one checks
that for each pair $(\alpha, \beta)$ with $1 \le \alpha < \beta \le 4$
$$\sum_{1 \le i < j \le 4} \epsilon_{ij}\frac{\partial f_{ij}}{x_{\alpha
    \beta}}\cdot f_{kl} = 0$$
for suitable signs $\epsilon_{ij}$. 
We may use this to construct a free $SG$-resolution 
$$0 \to S \xrightarrow{M} S^6 \to S \otimes_k V^\ast \to
\Omega^1_{S/A} \to 0$$
where the entries of $M$ generate $I_Z$ and $V$ is the vector space of $2 \times
4$ matrices. In particular  
$$\Ext^2_S(\Omega^1_{S/A}, S)^G \simeq (S/I_Z)^G = A/\mathfrak{m} = k$$
and indeed condition (ii) of Theorem \ref{defGtheorem} fails.
\end{example}

We assume now that $G$ is a \emph{quasitorus} (also called a diagonalizable group). See e.g.\  \cite[Section 1.2]{adhl:cox} for details and proofs of the statements below. We recall the definition

\begin{definition} \label{quasit} A quasitorus is an affine algebraic group $G$ whose algebra of regular functions $\Gamma(G,\CO_G)$ is generated as a $k$-vector space by the characters $\chi: G \to k^\ast$. A torus is a connected quasitorus.
\end{definition}
One proves that a quasitorus is a direct product of a torus and a finite abelian group. It is also characterized by the fact that any rational representation of $G$ splits into one-dimensional subrepresentations. In particular in our case the adjoint
action is trivial so $\CF \otimes \mathfrak{g} \simeq \CF^r$ as $\CS
G$-modules. 

\begin{lemma} \label{Qlemma} If $G$ is a quasitorus and $\mathcal{G}$
  is an $\CS G$-module 
  then there are inclusions $$\Hom_\CS(\CQ,
\mathcal{G})^G \subseteq \ \Hom_{\CO_X}(\CQ^G,
\mathcal{G}^G) \subseteq  (\mathcal{G}^G)^r \, .$$
\end{lemma}
\begin{proof}
Since $G$ is  quasitorus $\CS \otimes \mathfrak{g} \simeq \CS^r$ as $\CS
G$-modules and the globally generated free sheaf $\CS^r$ is generated by invariants. Thus $\CQ$ is generated by $G$ invariants, so
$$\Hom_\CS(\CQ,
\mathcal{G})^G \to\Hom_{\CO_X}(\CQ^G,
\mathcal{G}^G)$$
 is injective. The second inclusion follows
from first taking invariants of \eqref{euseq} and then applying
$\mathcal{H}om_{\CO_X}(-, \mathcal{G}^G)$.
\end{proof}

The results become better if  the codimension of
$Z(\CQ)$ is sufficiently large. 

\begin{proposition}\label{ZQ}  If $G$ is a quasitorus and $$\depth_Z\CF
  \ge 2, \quad \depth_Z\CF^G \ge 3 \text{ and } 
\depth_{Z(\CQ)} \CF^G \ge 4$$  then there is an exact sequence
$$\ 0 \to r H^2(X, \CF^G) \to T^2(\CS/\CO_X; \CF)^G
  \to \Hom_\CS(\Omega^1_{\CS/\CO_X}, \mathcal{H}^2_Z(\CF))^G \to r
  H^3(X, \CF^G) $$ 
\end{proposition}
\begin{proof} From 
Proposition \ref{h30} and Lemma \ref{Qlemma}  it is enough to
prove that $$\mathcal{H}om_{\CO_X}(\CQ^G, 
\mathcal{H}^3_Z(\CF^G)) = 0$$ but this follows from the assumptions and
the following lemma.
\end{proof}

\begin{lemma} Let $A$ be a noetherian ring, $M$ a finitely generated
  $A$-module, $I \subseteq J$ two ideals of $A$ and suppose $\depth_I
  M \ge d$ and $\depth_J M \ge d+1$. If $Q$ is an $A$-module with
  support contained in $V(J)$ then $\Hom_A(Q, H^d_I(M)) = 0$.
\end{lemma}
\begin{proof} Consider the two spectral sequences 
$$E_2^{p,q} =
  H^p_I(\Ext_A^q(Q,M)) \quad \text{and} \quad ^{\prime}E_2^{p,q} =
  \Ext_A^p(Q, H^q_I(M))$$ 
which both converge to $\Ext^{p+q}_{A,I}(Q,M)$. We have $H^q_I(M)=0$
for $q < d$ so the second spectral sequence yields $\Hom_A(Q,
H^d_I(M)) \simeq \Ext_{A,I}^d(Q,M)$. On the other hand since $\depth_J M \ge
d+1$ we have $\Ext^q_A(Q,M) = 0$ for $q < d+1$. Thus by the first
spectral sequence $\Ext_{A,I}^d(Q,M) = 0$, cf. \cite[Exp. VII]{gr:sga2}, Lemma 1.1.
\end{proof}

\begin{theorem}\label{cor1} Assume $G$ is a quasitorus and let
  $(X,Z,S,J)$ be a $G$-quadruple with associated sheaf of algebras $\CS$ and assume $\depth_Z\CS \ge 2$. If
\begin{itemize} 
\item[\textup{(i)}] $\depth_Z\CO_X \ge 3$ and  $\depth_{Z(\CQ)}\CO_X \ge 4$
\item[\textup{(ii)}] $\Hom_\CS(\Omega^1_{\CS/\CO_X},
  \mathcal{H}^2_Z(\CS))^G = 0$
\item[\textup{(iii)}] $G$ is finite or $H^2(X, \CO_X) = 0$
\end{itemize} 
and \textup{(iv)} of Theorem \ref{defGtheorem}
hold, then $\Def^G_{S} \to \Def_{X}$ is smooth. 
If moreover
\begin{itemize}
\item[\textup{(v)}] $G$ is finite or $H^1(X, \CO_X) = 0$
\end{itemize}
and \textup{(vi)} of Theorem \ref{defGtheorem} hold
then  $\Def^G_{S}
\to \Def_{X}$  is an isomorphism.  
\end{theorem}
\begin{proof} Apply Lemma \ref{defGlemma}, Proposition \ref{h30} and
  Proposition \ref{ZQ}. \end{proof}

If $S$ is regular then
we get a statement about the rigidity of $X$ and more generally about
rigidity of $X$ along a sheaf $\CF$, i.e.\ when $T^1_X(\CF)=0$. Note
that if $(X, Z , S, 
\bar{J})$ is a $G$-subquadruple of  $(Y,W,R,J)$ as defined in
Definition   \ref{subferrydef} then
the vanishing of $T^1_Y(f_\ast\CO_X)$ implies that the forgetful map
$\Hilb_{X/Y} \to \Def_X$ is smooth. This is often  useful for proving
unobstructedness. 

\begin{corollary}\label{cor2} Assume $G$ is a quasitorus and let
  $(X,Z,S,J)$ be a $G$-quadruple 
  with $S$ a regular ring, $M$ a finitely generated $SG$-module and
  $\CF = \pi_\ast(\widetilde{M}_{|X^\prime})$. If
\begin{list}{\textup{(\roman{temp})}}{\usecounter{temp}}
\item $\depth_Z \CF \ge 2$,  $\depth_Z \CF^G \ge 3$ and
  $\depth_{Z(\CQ)} \CF^G \ge 4$
\item $G$ is finite or $H^2(X, \CF^G) = 0$
\item $\Hom_\CS(\Omega^1_{\CS/\CO_X}, \mathcal{H}^2_Z(\CF))^G = 0$
\item $\Der_k(S, H^2_J(M))^G = 0$
\end{list} then   $T^1_X(\CF) = 0$. In particular if the conditions
hold for $M=S$ then $X$  is rigid.
\end{corollary}
\begin{proof} This follows from Lemma \ref{Mlemma},  Lemma
  \ref{lclemma} and Proposition 
  \ref{ZQ}. \end{proof} 

\begin{remark} By the Hochster-Roberts theorem $X$ will be
  Cohen-Macaulay and equidimensional if  $S$ is a
regular ring so we may exchange depth with codimension if $M = S$. 
\end{remark}

\begin{example}  \emph{Weighted projective space.} Consider now $X =
  \PP(\mathbf{q}) = \PP(q_0, \dots , q_n)$ as described in Example
  \ref{wps}. We use Corollary \ref{cor2} to prove that \emph{if no
$n-1$ of the $q_0, \dots , q_n$ have a common factor then
$\PP(\mathbf{q})$ is rigid}. 

The subscheme $Z$ is the singular locus of $X$ and the condition means that
$\codim Z \ge 3$. 
Let  $S= k[x_0, \dots ,x_n]$ with $n \ge 2$. We have
$$J^\prime = \bigcap_{\substack{p \text{ prime}
\\ p \mid \lcm(q_0, \dots , q_n)}} (x_i : p \nmid q_i)$$ 
so in fact $\codim J^\prime \ge 3$.
The sheaf $\CQ$ is trivial since
the isotropy is  finite everywhere. This takes care of
condition (i). 

The cohomology $H^2(X, \CO_X) = 0$ and since $J = (x_0, \dots , x_n)$
clearly $H^2_J(S)=0$. We are left with (iii) but as in Example
\ref{fG} since locally the quotient is by a finite group on a smooth
space, $\depth_Z \CS = \codim Z$ and $\mathcal{H}^2_Z(\CS) = 0$.
\end{example}

\section{Applications to toric varieties} \label{tvsec} We collect
here some results 
for toric varieties $X_\Sigma$ over $\CC$ that illustrate the various
aspects of our comparison theorems. We consider the $G$-quadruple
$(X_\Sigma,\Sing(X_\Sigma),S,B(\Sigma))$ described in Example
\ref{TVex} and will use the notation defined there. Note that toric
varieties are Cohen Macaulay and $S$ is a
polynomial ring. Moreover the condition $\depth_Z \CS \ge 2$ is always
satisfied since $U^\prime = \Spec S \setminus Z(\Sigma^\prime)$ where
$\Sigma^\prime$ is the fan of smooth cones in $\Sigma$ and $\codim
Z(\Sigma^\prime) \ge 2$.

\subsection{Subschemes of  simplicial toric varieties}
As explained before Corollary \ref{TVcor} a subscheme of a toric variety yields a $G$-subquadruple induced by a homogeneous $I$ in
the Cox ring $R$ of the toric variety $Y=Y_\Sigma$. Let $X
\subset Y$ be the subscheme. We have $Z = \Sing(Y) \cap
X$, $S = R/I$ and $J= 
(I+B(\Sigma))/I$. 

We  assume for simplicity that $Y$ is 
simplicial and all maximal cones have dimension $d$ and that $X$ is
Cohen-Macaulay and equidimensional. This means 
that the quotient $$\pi: Y^\prime = \Spec R \setminus V(B(\Sigma)) \to
Y$$ is a geometric quotient with all isotropy finite. In particular
$\CQ = 0$. 
Moreover locally on a chart $U_\sigma \subset Y$, corresponding to a
maximal cone in $\Sigma$, the quotient map sits
in a commutative diagram
\begin{equation*}
\begin{tikzcd} \CC^d \arrow{d}
\arrow[hookrightarrow]{r} & \CC^d \times (\CC^\ast)^{r}
\arrow{d}{\pi} \\ \CC^d/G_\sigma \arrow{r}{\simeq} &
U_\sigma 
\end{tikzcd}
\end{equation*} 
where $G_\sigma$ is finite abelian. See e.g.\ the proof of
\cite[Theorem 1.9]{bc:ho}. If $Z$ is closed in $X \subseteq Y$ then $\codim
\pi^{-1}(Z) = \codim Z$.

Since $\CQ = 0$ we may apply Theorem \ref{Q0thm}, but we still
need to ensure that $$\Hom_\CS(\mathcal{T}or(\Omega^1_{\CS/\CO_X}),
\mathcal{H}^2_Z(\CS))^G = 0 \, .$$
Because  $\codim \pi^{-1}(Z) = \codim Z$ this will be the case  if
$\codim Z \ge 3$. 
\begin{theorem}\label{THthm} If $X$
is an equidimensional  Cohen-Macaulay subscheme of a simplicial toric
variety $Y$ and 
\begin{list}{\textup{(\roman{temp})}}{\usecounter{temp}}
\item $\codim Z \ge 3$
\item $H^2(X, \CO_X) = 0$
\item $T^1(S/k; H^1_J(S))_0 = 0$ and $\Der_k(S, H^2_J(S))_0 = 0$
\end{list} then $\Def^0_{S} \to \Def_{X}$ is smooth. If moreover
\begin{list}{\textup{(\roman{temp})}}{\usecounter{temp}}
\item $H^1(X, \CO_X) = 0$
\item $\Der_k(S, H^1_J(S))_0= 0$
\end{list} then $\Def^G_{S}
\to \Def_{X}$  is an isomorphism.  
\end{theorem}

We give now two often studied situations where the theorem
applies. First, let $X = \Proj S$ where  $S$ is a finitely generated
$\ZZ_{+}$ graded algebra as in Example \ref{wps}. 

\begin{corollary} \label{wpscor2} Let $X = \Proj S$ be an equidimensional
  Cohen-Macaulay subscheme of a well formed 
  weighted projective space $\PP(q_0, \dots , q_n)$ defined by the
  homogeneous ideal $I$. Let $\mathfrak{m}$ be the irrelevant maximal ideal of
  $S$ and $Z$ the intersection of the singular 
  locus of $\PP(\mathbf{q})$ with $X$. Assume  $\codim Z \ge 3$ and
  $\depth_{\mathfrak{m}} S \ge 2$. If
$H^2(X, \CO_X) = 0$ and $H^1(X, \CO_X(q_i)) = 0$ for all $i= 0, \dots
,n$ then $\Def^0_{S} \to \Def_{X}$ is smooth. If moreover
$H^1(X, \CO_X) = 0$ then  $\Def^G_{S}
\to \Def_{X}$  is an isomorphism.  
\end{corollary}

Secondly let $X$ be a Calabi-Yau hypersurface in a simplicial Gorenstein
Fano toric variety $Y$. (See the survey book \cite{ck:mir} or the original
paper by Batyrev \cite{ba:dua} for details.) Thus $Y$ is given by the
normal fan of a simple reflexive polytope and $X$ is a divisor in  the
class of $-K_Y$ and therefore ample and Cartier. In particular
$\omega_X = \CO_X$ and $H^i(X, \CO_X)= 0$ for $i \ge 1$. We refer to
it as a Calabi-Yau hypersurface even though it may be highly singular. 

Let $D_i$ be the divisors corresponding to the rays $\rho_i$. The
hypersurface $X$ is defined by some $f \in R_\beta$ where $\beta =
\sum D_i$. For our results it is not necessary to assume any
generality for $f$. We will need the following lemma and are grateful
to Benjamin Nill for supplying the proof. (We use standard toric
geometry notation as found in e.g.\  \cite{cls:tor}.)

\begin{lemma}\label{ben} Let $P$ be a simple reflexive lattice
  polytope with inward normal fan $\Sigma$ and let $Y$ be the
  corresponding simplicial Gorenstein Fano toric variety. For every
  ray $\rho_i$ in $\Sigma$ the $\QQ$-Cartier divisor $E_i = \sum_{j
    \ne i} D_j = -K_Y - D_i$ is nef and big. 
\end{lemma}
\begin{proof} Let $P^*$ be the dual reflexive polytope and $v_i$ the
  primitive lattice points on $\rho_i$ (i.e., the vertices of
  $P^*$). Let $h$ be the 
piecewise linear function on $N_\RR$ such that $h(v_i) = 0$  and
$h(v_j) = -1$ for $j \ne i$. Then $E_i$ is the divisor associated to
$h$ and $P_h = \{m \in M_\RR: \langle m, v_i \rangle \ge 0 \text{ and }
\langle m, v_j \rangle \ge -1 \text{ for } i \ne j\}$ is the
corresponding polytope. Geometrically, $P_h$ is $P$ after moving the facet
$F_i=(v_i)^*$ of $P$ one step inwards (so that the origin now lies on the
boundary of $P_h$). Thus $\dim P_h = \dim P$ and $E_i$ is big.

Let $\sigma$ be a maximal cone of $\Sigma$, and $\sigma^*$ the corresponding
vertex of $P$. Since $\Sigma$ is simplicial, $E_i$ is $\QQ$-Cartier so
there exists $l_\sigma \in M_\QQ$ with $\langle l_\sigma, v \rangle = h(v)$
for $v \in \sigma$. It is well known that if $l_\sigma$ is contained
in $P_h$ for all $\sigma$ then $E_i$ is nef.

There are two cases, $v_i \not\in \sigma$ or $v_i \in \sigma$.
If $v_i \not\in \sigma$, then $l_\sigma$ evaluates to $-1$ on any vertex
of $\sigma$. In other words, $l_\sigma = \sigma^* \in P$. Thus
$l_\sigma$ evaluates $\ge -1$ on any $v_j$. Moreover, since
by duality $l_\sigma=\sigma^* \not\in F_i$, we get that $l_\sigma$
evaluates $> -1$ on $v_i$. However, P is a lattice polytope, hence 
$l_\sigma$ as a vertex of $P$ is a lattice point, so 
$l_\sigma$ evaluates $\ge 0$ on $v_i$.

Assume $v_i \in \sigma$. Since $\Sigma$ is simplicial, the facet of $P^*$
corresponding to the maximal cone $\sigma$ is the convex hull of $v_i$ and
a $(d-2)$-dimensional face $G$ of $P^*$. Note that by duality $G^*$ is
an edge of $P$, and $\sigma^*$ is contained in
$G^*$ and the facet $F_i=(v_i)^*$. In particular, the intersection of the affine
hull of $G^*$ with the hyperplane $H$ orthogonal to $v_i$ (which is parallel
to $F_i$) is precisely the point $l_\sigma$. Since $G^*$ is an edge of $P^*$, it has a
vertex $w$ different from $\sigma^*$. Since $w$ is not in $F_i$ (otherwise, $G^*$
would be contained in $F_i$, hence $v_i$ would be contained in $G$), we have
that $v_i$ evaluates $> -1$ on $w$. Because $P$ is a lattice polytope, $w$ is a lattice
point, so $v_i$ evaluates $\ge 0$ on $w$. In particular, $l_\sigma$ (as it lies
on the affine hull of $G^*$ and evaluates to $0$ with $v_i$) lies between
$\sigma^*$ and $w$. Therefore, $l_\sigma \in P$. This implies that not only
$l_\sigma$ evaluates to $0$ on $v_i$, but it evaluates $\ge -1$ with all other
$v_j$'s, as desired.
\end{proof}

\begin{theorem} Let $X$ be a Calabi-Yau hypersurface in a simplicial
  Gorenstein 
Fano toric variety $Y$ defined by $f \in R_\beta$ where $R=\CC[x_1,
\dots , x_N]$ is the Cox
ring of $Y$ and $\beta$ is the anti-canonical class. If $\dim Y
\ge 3$ and $\codim (X \cap \Sing(Y)) \ge 3$ in $X$ then $\Def_X$ is
smooth and its tangent space is isomorphic to the degree $\beta$ part
of the Jacobian algebra of $f$. That is
$$T^1_X \simeq \left( \CC[x_1,
\dots , x_N]/(\frac{\partial f}{\partial x_1}, \dots ,
  \frac{\partial f}{\partial x_N}) \right)_\beta \, .$$ 
\end{theorem}
\begin{proof} We try to apply Theorem \ref{THthm}. Since $H^i(X,
    \CO_X) = 0$ for $i=1,2$ by the 
  Calabi-Yau property, it suffices to show that $H^1_J(S)_\beta$, $H^2_J(S)_{D_i}$
  and $H^1_J(S)_{D_i}$ all vanish. From the exact sequence
$$0 \to R(-\beta) \to R \to S \to 0$$ we see that this would follow
from 
\begin{list}{\textup{(\roman{temp})}}{\usecounter{temp}}
\item $H^1_J(R)_\beta = 0$ and $H^1_J(R)_{D_i}=0$
\item $H^2_J(R)_0 \simeq H^1(Y, \CO_Y) = 0$
\item $H^2_J(R)_{D_i} \simeq H^1(Y, \CO_Y(D_i)) = 0$
\item  $H^{j+1}_J(R)_{D_i - \beta} \simeq H^j(Y, \CO_Y(K_Y +D_i)) = 0$ for $j=1,2$.
\end{list} 
Now (i) is true because $\depth_J(R) \ge 2$ and (ii) is true for all
complete toric varieties. To prove (iii) we have that $-K_Y$ is ample
so we may e.g.\ apply a vanishing theorem of Musta{\c{t}}{\v{a}} as
stated in \cite[Theorem 9.3.7]{cls:tor}. The last vanishing follows from the
assumption that $\dim Y \ge 3$, Lemma \ref{ben} and the $\QQ$-Cartier
version of a vanishing result
of Batyrev-Borisov as stated in \cite[Theorem 9.3.5]{cls:tor}.
\end{proof}

\subsection{Local cohomology computations}\label{lcc}  In the following it will be
important to be able to compute with the modules 
$H^i_B(S)$ where $S$ is the Cox ring of a toric variety and $B$ the
irrelevant ideal for some fan. There is a combinatorial method due to
Musta{\c{t}}{\v{a}} and we will make some simplifications in the case
$i=2$.

Let $\{m_1, \dots , m_s\}$ be monomial generators for any
squarefree monomial ideal $B \subseteq S$. For $I \subseteq \{1, \dots
, s\}$ let $T_I$ be the simplicial complex on the vertex set $\{1,
\dots , s\}$ where $\{j_1, \dots , j_k\}$ is a face if $x_i \nmid
\lcm(m_{j_1}, \dots , m_{j_k})$ for some $i \in I$. If $p \in \ZZ^m$
define $\negative(p) \subseteq \{1, \dots , m\}$ to be the set $\{i
\mid p_i \le -1\}$. Let $\{e_1, \dots , e_m\}$ be the standard
generators of $ \ZZ^m$.

\begin{theorem}\emph{(\cite[Theorem 2.1]{mu:lo})} \label{mu} If $p \in
\ZZ^m$ then there are isomorphisms $H^i_B(S)_p \simeq
\widetilde{H}^{i-2}(T_I; k)$ when $I = \negative(p)$. Moreover the map
$$H^i_B(S)_p \xrightarrow{\cdot x_i} H^i_B(S)_{p+e_i}$$
corresponds to the map $H^{i-2}(T_{\negative(p)} ; k) \to
H^{i-2}(T_{\negative(p+e_i)} ; k)$ induced in cohomology by the
inclusion $T_{\negative(p+e_i)} \subseteq T_{\negative(p)}$.  In
particular, if $p_i \ne -1$, then $\cdot x_i$ is an isomorphism.
\end{theorem}

From now on assume $S$ is the Cox ring and $B$ the irrelevant ideal
for a fan $\Sigma$.

\begin{lemma} The codimension 2 prime ideals of $B$ are the $(x_i,
x_j)$ with $\rho_i$ and $\rho_j$ not in the same cone in $\Sigma$.
\end{lemma}
\begin{proof} This follows directly from the description of the prime
components of $B$. See e.g.\ \cite[Proposition 5.1.6]{cls:tor}.
\end{proof}

Let $B_2$ be the intersection of the codimension 2 primes of $B$. Let
$K$ be the simplicial complex which has $B_2$ as Stanley-Reisner
ideal. Let $\Gamma$ be the graph with vertices $\{0, \dots , m\}$ and
edges $\{i,j\}$ when $\rho_i$ and $\rho_j$ are in the same cone in
$\Sigma$. Define $C(\Gamma)$ to be the clique complex of $\Gamma$. Let
$\Gamma_I$ be the induced subgraph with vertices in $I$.

\begin{lemma} \label{TI1} The complex $K$ is the Alexander dual of
$C(\Gamma)$. In particular there is a one to one correspondence
between facets of $C(\Gamma)$ and a minimal generating set for $B_2$
given by $F \mapsto x_{F^c}$. This correspondence identifies $T_I$
with the complex with vertex set equal the set of facets of
$C(\Gamma)$ containing an element of $I$ and $\{F_{i_1}, \dots ,
F_{i_k}\}$ is a face if $i \in \bigcap F_{i_j}$ for some $i \in I$.
\end{lemma}
\begin{proof} The facets of $K$ are the complements $\{i,j\}^c$ where
$\{i,j\}$ is a non-edge of $\Gamma$. We have $f \in C(\Gamma)^\vee
\Leftrightarrow f^c \notin C(\Gamma) \Leftrightarrow f^c$ contains a
non-edge of $\Gamma$ $\Leftrightarrow f \subseteq \{i,j\}^c$ for a
non-edge $\{i,j\}$ $\Leftrightarrow f \in K$. A set $\{x_{F_{i_1}^c},
\dots , x_{F_{i_k}^c}\}$ is a face of $T_I$ if there is an $i \in I$
with $i \notin \bigcup F_{i_k}^c = \left( \bigcap F_{i_k} \right)^c$,
that is if $i \in \bigcap F_{i_k}$.
\end{proof}

\begin{lemma} \label{TI2} If $i \in I$ let $f_i$ be the face of $T_I$
given by $ \{F :i \in F\}$.  The map $H_0(\Gamma_I) \to H_0(T_I)$
defined by $[i] \mapsto [F]$ where $F$ is any vertex in $f_i$ is an
isomorphism.
\end{lemma}
\begin{proof} First note that $\{i,j\}$ is an edge in $\Gamma_I$ iff
there is a maximal clique $F$ of $\Gamma$ containing $\{i,j\}$ which
is iff $f_i \cap f_j \ne \emptyset$. To show that the map is well
defined assume $[i] = [j]$. Then there is an edge path through the
vertices $i = i_0, i_1, \dots ,i_r = j$ in $\Gamma_I$. Then since the
$f_{i_k} \cap f_{i_{k+1}} \ne \emptyset$ we may find a path from any
vertex of $f_i$ to any vertex in $f_j$.

The map is clearly surjective. To show injectivity assume $F \in f_i$
and $G \in f_j$ are connected by an edge path $F = F_0, F_1, \dots
,F_r = G$. This means that the intersection of cliques $F_k \cap
F_{k+1}$ contains an element $i_k \in I$. This implies that there are
edges $\{i_k, i_{k+1}\} \in \Gamma$ and therefore in the induced
subgraph $\Gamma_I$. Thus $[i] = [j]$.
\end{proof}

\begin{proposition} There is an isomorphism $H^2_B(S)_p \simeq
\widetilde{H}^0(\Gamma_I, k)$ where $I = \negative(p)$.
\end{proposition}
\begin{proof} This follows now directly from Theorem \ref{mu}, Lemma
\ref{TI1} and \ref{TI2}.
\end{proof}

\subsection{$T^1$ for toric singularities} There has been much interest in
computing deformations of affine toric varieties. In particular there
are several combinatorial descriptions of these due to Altmann, see
e.g.\ \cite{al:one} and the references therein. We give here an
alternative description using the Cox ring and compute $T^1$ in some
special cases. To simplify matters we will assume  $\codim \Sing(X) \ge 3$.

Let $X$ be an affine toric variety
with cone $\sigma$ and Cox ring $S$. In the grading defined by the
abelian group 
$\Cl(X)$ we have $X = \Spec S_0$.  Set $Z = \Sing(X)$.  Thus $(X, Z,
S, (1))$ is an affine $G$-quadruple where $G = \Hom_\ZZ(\Cl(X),
\CC^\ast)$. The associated sheaf of algebras is just $S$ itself. If $\Sigma$ is the fan
consisting of smooth faces of $\sigma$ and $B=B(\Sigma)$ the
corresponding irrelevant ideal then the sheaves $\mathcal{H}_Z^i(\CS)$
correspond to the modules $H^i_B(S)$.

Let $n = \dim X = \rank N$, $m = |\Sigma(1) |$ and $r = \rank \Cl(X) =
m-n$. Let $C \subseteq \Cl(X)$ be the free part of the abelian
group. Fix an isomorphism $C \simeq \ZZ^r$ so that the map $\ZZ^m \to
C$ has matrix $A = (a_{ij})$. The $\ZZ^r$ part of the $\Cl(X)$ grading
on $S=\CC[x_1, \dots , x_m]$ is given by the columns of $A$, i.e.\ $\deg x_j = (a_{1j}, \dots ,
a_{rj})$. A set of Euler derivations as in  Definition \ref{eudef}
is $$\{ E_i = \sum_j a_{ij} x_j \frac{\partial}{\partial x_j} : i = 1,
\dots , r\} \, .$$
Thus the module $Q$ corresponding to the sheaf $\mathcal{Q}$ in
Definition \ref{eudef} has a graded  presentation 
\begin{equation}\label{qprestv} \bigoplus_{j=1}^m  S(- \deg x_j) \xrightarrow{\tilde{A}} S^r \to Q
\to 0\end{equation}
where $\tilde{A} = (a_{ij}x_j)$. The support of $Q$ on $X$ is the
non-simplicial locus corresponding to the irrelevant ideal $B(Q)
\subset S$ for the fan consisting of simplicial faces of $\sigma$. 

\begin{theorem} \label{afftor} Let $X$ be an affine toric variety.
\begin{list}{\textup{(\roman{temp})}}{\usecounter{temp}}
\item If $X$ is simplicial then $\Hom_S(\Omega^1_{S/S_0}, H^2_B(S))_0 \simeq T^1_X$. 
  \item If $\codim \Sing(X) \ge 3$, then there is an exact sequence 
$$0 \to \Hom_S(\Omega^1_{S/S_0}, H^2_B(S))_0 \to T^1_X \to
\Hom_S(Q, H^3_B(S))_0 \to 0$$
and inclusions $\Hom_S(Q, H^3_B(S))_0 \subseteq \ \Hom_{S_0}(Q_0,
H^3_{B_0}(S_0)) \subseteq H^3_{B_0}(S_0)^r$. 
\end{list} 
\end{theorem}
\begin{proof} As in the proof of
  Lemma \ref{sit1},  $T^1_X = T^1_{S_0} = 
  T^1_{S_0}(S)_0$. The Zariski-Jacobi sequence for $\CC \to S_0 \to
  S$ and the fact that $S$ is regular yields $T^1_{S_0}(S) \simeq
  T^2_{S/S_0}$. Furthermore by Lemma \ref{tiz} we have $ T^2_{S/S_0}
  \simeq \Ext_S^2(\Omega_{S/S_0}^1, S)$. The result, except for the
  $0$ on the right in the exact sequence in (ii), now follows by
applying the proof of Proposition \ref{h30} to the affine
  case. 

Recall from the proof of Proposition \ref{h30}
that the cokernel of $\Hom_S(\Omega^1_{S/S_0}, H^2_B(S)) \to
\Ext_S^2(\Omega_{S/S_0}^1, S)$ is the kernel $K$ of $H^3_{B}(S)^r \to
\Ext_{S, B}^3(\Omega_{S/S_0}^1, S)$ in the long exact sequence for
$\Ext_{S, B}^i(\Omega_{S/S_0}^1, S)$. We have a commutative
diagram with exact rows and columns
\begin{equation*}
\begin{tikzcd} 
&&& \Ext_S^1(\Omega_{S/S_0}^1,  H^2_B(S)) \arrow{d}
\\ 0 \arrow{r} & K \arrow{r} \arrow{d} & H^3_{B}(S)^r \arrow{d}{=}
\arrow{r}{u} & \Ext_{S, B}^3(\Omega_{S/S_0}^1, S)
\arrow{d}{v} \\ 0 \arrow{r} & \Hom_S(Q,
  H^3_B(S)) \arrow{r} &H^3_{B}(S)^r \arrow{r} &
\Hom_S(\Omega^1_{S/S_0}, H^3_B(S))  
\end{tikzcd}
\end{equation*} 
where the right column comes from the spectral sequence for $\Ext_{S,
  B}^i(\Omega_{S/S_0}^1, S)$. By assumption $H^2_B(S)_0 = 0$ so by the
below Lemma \ref{Extsm}, $\Ext_S^1(\Omega_{S/S_0}^1,  H^2_B(S))_0 = 0$
and $K \simeq \Hom_S(Q,H^3_B(S))$.
\end{proof}

\begin{lemma} \label{Extsm} Let $G$ be a linearly reductive group
acting on an $n$-dimensional $k$ vector space $V$. Let $A = \Sym(V)$
and let $N$ be an $AG$-module with $N^G = 0$. Then
$\Ext^1_{A}(\Omega^1_{A/A^G}, N)^G = 0$.
\end{lemma}
\begin{proof} If $$A^G =
k[\varphi_1, \dots , \varphi_m] \subseteq A = k[x_1, \dots , x_n]$$
then the beginning of an $AG$ projective resolution of
$\Omega^1_{A/A^G}$ looks like
$$ \dots \to A^m \xrightarrow{\left(\frac{\partial \varphi_j}{\partial
      x_i} \right)} A \otimes_k V^{\ast} \to \Omega^1_{A/A^G} \to 0 \,
.$$ 
If $M$ is the image of $A^m$ in $A \otimes_k V^{\ast}$ then $M$ is
generated by invariants, so $\Hom_A(M,N)^G = 0$. Thus
$\Ext^1_{A}(\Omega^1_{A/A^G}, N)^G = 0$ by the long exact sequence for $\Ext$.
\end{proof}

We can ask when the contribution from $\Hom_S(\Omega^1_{S/S_0},
H^2_B(S))_0 = \Der_{S_0}(S, H^2_B(S))_0$ vanishes. This would happen
if $H^2_B(S)_{\alpha_i} = 0$ for $i = 1, \dots , m$ where $\alpha_i$
is the degree of $x_i$.
Lifting this to the $\ZZ^m$ grading we need $H^2_B(S)_{p} = 0$ when $p
= q + e_i$ and $q = (\dots 
, \langle u, v_j\rangle , \dots)$ for some $u \in M$.  If we assume
$\codim \Sing(X) \ge 3$ then $H^2_B(S)_0 = H^2_{\Sing(X)}(S_0) = 0$. In
the $\ZZ^m$ grading this means $H^2_B(S)_q = 0$ when $q$ is as
above. Thus by Theorem \ref{mu} it is enough to check when the
coordinate $\langle u, v_i\rangle$ of $q$ equals $-1$.

Let $\Gamma^f$ be the graph with vertices $\{0, \dots , m\}$ and edges
$\{i,j\}$ when $\{\rho_i, \rho_j\}$ span a 2-dimensional face of
$\sigma$ and $\Gamma^f_I$ the induced subgraph as above. If $\codim
\Sing(X) \ge 3$ then $\Gamma^f_I$ is a subgraph of $\Gamma_I(\Sigma)$
with the same vertex set so the map $\widetilde{H}_0(\Gamma^f_I) \to
\widetilde{H}_0(\Gamma_I)$ is surjective.
If $u \in M$ and $\langle u, v_i \rangle = -1$ let
$$I_i(u) =  \{j \in \{1, \dots , m\} : \langle u, v_j \rangle \le -1\}
\setminus \{i\} $$ and set $\Gamma_i(u) = \Gamma^f_{I_i(u)}$. (These graphs
also appear in the study of deformations of smooth complete toric
varieties in \cite{il:def}.) Thus we have

\begin{proposition}\label{rigid} Let $X$ be an affine toric variety
  that  is smooth in codimension 2 and 
$\Gamma_i(u)$ is connected for all $u \in M$. Then $\Hom_S(\Omega^1_{S/S_0},
H^2_B(S))_0 = 0$. 
\end{proposition}

The connectivity of these graphs can be proven if there is a polytope
we can use. For this we need the following lemma.
\begin{lemma} \label{Plemma} Let $P$ be a polytope, $H^\ast$ a closed
half-space with bounding hyperplane $H$ and $v$ a vertex of $P$ with
$v \in H$. If $\Gamma_v(H)$ is the induced subgraph of the edge graph
of $P$ on the vertices in
$$\{w \in \vertices P : w \in H^\ast\} \setminus \{v\}$$
then $\Gamma_v(H)$ is connected.
\end{lemma}
\begin{proof} If $H \cap \inter P \ne \emptyset$ let $H^\prime$ be the
supporting hyperplane in $H^\ast$ that is parallell to $H$. By
standard arguments every vertex in $\Gamma_v(H)$ is either on
$H^\prime$ or has a neighbor vertex in $H^\ast$ that is nearer to
$H^\prime$. Thus $\Gamma_v(H)$ is connected in this case.

If $H \cap \inter P = \emptyset$ then $H$ is a supporting
hyperplane. If $P \subset H^\ast$ then $\Gamma_v(H)$ is the edge graph
of $P$ with one vertex removed. Since edge graphs of $d$-polytopes are
$d$-connected $\Gamma_v(H)$ is connected (trivially also for $d =
0,1$). If $P \cap H^\ast = F$ is a face then the same argument applied
to $F$ yields the result.
\end{proof}

\begin{proposition}\label{qG} A $\QQ$-Gorenstein affine toric variety
that is smooth in codimension 2 has $\Hom_S(\Omega^1_{S/S_0},
H^2_B(S))_0 = 0$.
\end{proposition}
\begin{proof} An affine toric variety is $\QQ$-Gorenstein if and only
if there is a primitive $u_0 \in M$ and a positive integer $g$ such
that $\langle u_0, v_i \rangle = g$ for all generators $v_i$ of the
rays of $\sigma$. This means that if $P$ is the convex hull of the
$v_i$ in the hyperplane $\langle u_0, -\rangle = g$, then for all $u
\in M$ the graph $\Gamma_i(u)$ is a $\Gamma_v(H)$ as in Lemma
\ref{Plemma}.
\end{proof}

\begin{example} We compute $T^1_X$ for a non-simplicial
  toric 
  3-dimensional Gorenstein isolated singularity. Note that it follows
  from Theorem \ref{afftor} and
  Proposition \ref{qG}  that a toric 
  $d$-dimensional $\QQ$-Gorenstein isolated singularity is rigid if $d
  \ge 4$. (For a complete
  description of $T^1$ for all toric $\QQ$-Gorenstein singularities
  see \cite{al:min} and \cite{al:one}.) From Theorem \ref{afftor} and
  Proposition \ref{qG} it follows that $T^1_X \simeq
  \Hom(Q,H^3_B(S))_0$.

The cone $\sigma$ has a special form as it comes from a plane lattice
polygon $P$ with smooth face fan (and more than 3 vertices). If $(\alpha_1, \beta_1), ... ,
(\alpha_m, \beta_m)$ are the vertices  
of such a $P$ then the primitive generators of the rays of $\sigma$
are $v_i = (\alpha_i, \beta_i, 1)$. 

The fan $\Sigma$ consists of 
all faces of $\sigma$ except $\sigma$ itself. Using the
description of the prime components in  \cite[Proposition
5.1.6]{cls:tor} one sees that $B = B_2$ in the notation Section
\ref{lcc}. Thus we may use the methods there to show that $H^3_B(S)_p \simeq
H^1(\Gamma_I, k)$ where $I = \negative(p)$. Now it is easy to see that
$H^3_B(S)_p \ne 0$ if and only if all $p_i \le -1$. Note that the
Gorenstein property can be seen through the one to one
correspondence $p \mapsto -p - (1, \dots, 1)$ between non-zero pieces
of $H^3_B(S)_0$ and $S_0$. 

In our situation $\Cl(X) \simeq \ZZ^r$ so there is an exact sequence
$$0 \to \ZZ^3 \xrightarrow{A^\prime} \ZZ^m \xrightarrow{A} \ZZ^r 
\to 0$$ where the rows of $A^\prime$ are $[\alpha_i, \beta_i, 1]$ and
$A = (a_{ij})$ gives the set of Euler derivations. Using the presentation
of $Q$ in \eqref{qprestv} we may describe the
$\ZZ^m$ graded pieces of $\Hom(Q,H^3_B(S))_0$ as follows. It is
convenient to use the above one-to-one correspondence so set $q = -p -
(1, \dots, 1)$. If
$\Hom(Q,H^3_B(S))_{0,p} \ne 0$ then $A\cdot q = 0$ and
$\negative(q) = \emptyset$. If this is satisfied then $\Hom(Q,H^3_B(S))_{0,p}$ is 
isomorphic to the kernel of the submatrix $A^t_q$ consisting of the
rows $A^t_i$ of $A^t$ where $q_i > 0$.

From the smoothness of $P$ it follows that \emph{all} the $3 \times 3$
minors of $A^\prime$ are non-zero. Thus $A$ has the special property:
\begin{equation}\label{min}\text{All $r \times r$ minors of
$A$ are non-zero.}
\end{equation} Let $k$ be the number of $i$ with $q_i \ne 0$. If $k
\ge r$ then \eqref{min} implies that  $A^t_q$ has trivial kernel. On the other
hand if $1 \le k < r$ let $q^\prime = (q_{i_1}, \dots , q_{i_k})$ be the vector
of non-zero components. Since $Aq = 0$ we must have $(A^t_q)^t
q^\prime = 0$, but this is impossible since by \eqref{min} the columns
of $(A^t_q)^t$ are independent.

We conclude that the only non-zero $\Hom(Q,H^3_B(S))_{0,p}$ is when $p
= (-1, \dots , -1)$ and that $\dim T^1_X = \dim \Hom(Q,H^3_B(S))_{0} =
r = m-3$.
\end{example}

\subsection{Rigidity results}
The above together with Corollary \ref{cor2} yield rigidity
results for toric singularities. 
First we reprove a theorem of Altmann in \cite{al:min}.
\begin{corollary}\label{klaus} A $\QQ$-Gorenstein affine toric variety
that is smooth in codimension 2 and simplicial in codimension 3 is
rigid.
\end{corollary}
\begin{proof} This follows directly from Corollary \ref{cor2} and
  Proposition \ref{qG}.
\end{proof}

With the same type of arguments we can reprove Totaro's generalization,
\cite[Theorem 5.1]{to:ju} of theorems of Bien-Brion and de
Fernex-Hacon.
\begin{corollary} A toric Fano variety that is smooth in codimension
2 and simplicial in codimension 3 is rigid.
\end{corollary}
\begin{proof} Corollary \ref{klaus} takes care of the local situation
so what is left to prove is condition (iv) in Corollary
\ref{cor2}. This follows from $H^2_{B(\Sigma)}(S)_{D_i} = H^1(X,
\CO_X(D_i) = 0$ since $X$ is Fano. Alternatively, the Fano condition
implies that the fan $\Sigma$  is the face fan of a polytope. Thus the
same arguments 
as above show that $H^2_{B(\Sigma)}(S)_{D_i} = 0$.
\end{proof}

\bibliographystyle{amsalpha}

\providecommand{\bysame}{\leavevmode\hbox to3em{\hrulefill}\thinspace}
\providecommand{\MR}{\relax\ifhmode\unskip\space\fi MR }
\providecommand{\MRhref}[2]{%
  \href{http://www.ams.org/mathscinet-getitem?mr=#1}{#2}
}
\providecommand{\href}[2]{#2}

\end{document}